\documentclass[a4paper,11pt]{amsproc}
\usepackage{amsmath,amsthm}
\usepackage{amssymb,esint}
\usepackage{amscd}
\usepackage{enumerate}
\usepackage{anysize}
\marginsize{3.5cm}{3.5cm}{2.5cm}{2.5cm}

\usepackage[]{hyperref}
\hypersetup{
    colorlinks=true,       % false: boxed links; true: colored links
    linkcolor=black,          % color of internal links
    citecolor=black,        % color of links to bibliography
    filecolor=magenta,      % color of file links
    urlcolor=cyan           % color of external links
}

\newtheorem{theorem}{Theorem}[section]

\newtheorem{corollary}[theorem]{Corollary}
\newtheorem{definition}[theorem]{Definition}

\newtheorem{lemma}[theorem]{Lemma}
\newtheorem{notation}[theorem]{Notation}
\newtheorem{proposition}[theorem]{Proposition}
\newtheorem{remark}[theorem]{Remark}

\DeclareMathOperator{\cnx}{div}

\DeclareMathOperator{\diff}{d}

\renewcommand\({\left(}
\renewcommand\){\right)}

\def\nc{\newcommand}
\def\be{\beta}

\def\lam{\lambda}

\def\ra{\rightarrow}
\def\la{\leftarrow}

\def\D{\la D\ra}

\nc\pa{\partial}

\nc\CC{\mathbb{C}}
\nc\RR{\mathbb{R}}
\nc\QQ{\mathbb{Q}}
\nc\ZZ{\mathbb{Z}}
\nc\NN{\mathbb{N}}

\def\ba{\begin{align}}
\def\bad{\begin{aligned}}
\def\be{\begin{equation}}
\def\ea{\end{align}}
\def\ead{\end{aligned}}
\def\ee{\end{equation}}
\def\e{\eqref}

\def\dalpha{\diff \! \alpha}

\def\dt{\diff \! t}
\def\dtau{\diff \! \tau}
\def\dh{\diff \! h}
\def\dr{\diff \! r}

\def\dx{\diff \! x}
\def\dxi{\diff \! \xi}

\def\dy{\diff \! y}
\def\dz{\diff \! z}

\def\fract{\frac{\diff}{\dt}}

\def\defn{\mathrel{:=}}
\def\eps{\varepsilon}
\def\la{\left\vert}
\def\lA{\left\Vert}
\def\bla{\big\vert}
\def\blA{\big\Vert}
\def\le{\leq}
\def\les{\lesssim}
\def\mez{\frac{1}{2}}
\def\ra{\right\vert}
\def\rA{\right\Vert}
\def\bra{\big\vert}
\def\brA{\big\Vert}
\def\tdm{\frac{3}{2}}

\def\xR{\mathbb{R}}
\def\xS{\mathbb{S}}

\def\a{\alpha}
\def\ca{\check{\a}}

\def\Lr{\mathcal{L}}

\begin{document}

\title{Quasilinearization of the 3D Muskat equation, and applications to the critical Cauchy problem}
\author{Thomas Alazard}
\author{Quoc-Hung Nguyen}
\date{}

\begin{abstract}
We exhibit a new decomposition of the nonlinearity for the Muskat equation and use it to commute Fourier multipliers with the equation. 
This allows to study solutions with critical regularity. 
As a corollary, we obtain the first well-posedness result for arbitrary large data in the 
critical space $\dot{H}^2(\xR^2)\cap W^{1,\infty}(\xR^2)$. Moreover, we prove the existence of solutions 
for initial data which are not Lipschitz.
\end{abstract}   

\maketitle

\section{Introduction}

The Muskat problem dictates the evolution of a time-dependent 
interface $\Sigma(t)$ separating 
two fluids in porous media obeying Darcy's law (\cite{darcy1856fontaines,Muskat}). 
We assume that the interface is a graph, 
so that at time $t\ge 0$, 
$$
\Sigma(t)=\{ (x,f(t,x))\,:\, x\in \xR^d\}\subset \xR^{d+1},
$$
with $d=2$ (resp.\ $d=1$) for a 3D fluid domain (resp.\ 2D). 
It has long been understood that the Muskat problem can be reduced 
to a parabolic evolution equation for the unknown function $f$  
(see~\cite{CaOrSi-SIAM90,EsSi-ADE97,PrSi-book,SCH2004}). 
C\'ordoba and Gancedo~\cite{CG-CMP} have 
studied this problem using contour integrals, and 
obtained a beautiful formulation of the Muskat equation 
in terms of finite differences, which reads
\be\label{Muskat}
\partial_tf(t,x)=\frac{1}{2^{d-1}\pi}\int_{\mathbb{R}^d}
\frac{\alpha\cdot\nabla_x\Delta_\alpha f(t,x)}{\langle \Delta_\alpha f(t,x)\rangle^{d+1}}\frac{\dalpha}{|\alpha|^{d}},\qquad d=1,2,
\ee
where $\Delta_\alpha f$ is the slope, defined by
\begin{equation*}
\Delta_{\alpha}f(t,x)=\frac{f(t,x)-f(t,x-\alpha)}{|\alpha|},
\end{equation*}
and where, for $d=2$, 
\begin{equation*}
\la\a\ra=\sqrt{\a_1^2+\a_2^2},\quad \langle \Delta_\alpha f\rangle=\sqrt{1+(\Delta_\alpha f)^2},\quad \a\cdot\nabla_x=\alpha_1\partial_{x_1}+\a_2\partial_{x_2},
\end{equation*}
with obvious modifications for $d=1$.

We are inspired by many recent works where 
this formulation is used to study 
the Cauchy problem: in particular 
Constantin, C\'{o}rdoba, Gancedo, 
Strain~\cite{CCGRPS-JEMS2013}, 
Constantin, Gancedo, Shvydkoy and Vicol~\cite{CCGRPS-AJM2016}, Matioc~\cite{Matioc2}, 
C{\'o}rdoba and Lazar~\cite{Cordoba-Lazar-H3/2}, Alazard and Lazar~\cite{Alazard-Lazar}, 
Gancedo and Lazar~\cite{Gancedo-Lazar-H2} (other references are given in \S\ref{S:1.2}). Inspired by the latter, we have 
studied in~\cite{AN1,AN2,AN3} the Cauchy problem for solutions of the 2D Muskat's equation with 
critical regularity. 
The main result of the previous series of papers is established in~\cite{AN3}. It reads 
that, for $d=1$, the Cauchy problem for the 
Muskat equation is well-posed on the endpoint Sobolev space $H^{\tdm}(\xR)$. 
Let us make it clear from the start that we are not going to extend this result to the 3D setting (and, according to the authors' understanding of the problem, it is not certain that this is possible).
The main goal of the present paper is to show that it is possible to solve the Cauchy 
problem for arbitrary large data in the 
critical space $\dot{H}^2(\xR^2)\cap W^{1,\infty}(\xR^2)$. Moreover, we prove the existence of solutions 
for initial data which are not Lipschitz.

A guiding principle in the study of free boundary 
problems with potential flows is that problems in 2D are simpler than those in 3D, 
because complex function theory can be used in the first case. Such a 
dichotomy is well understood  in the analysis of the water-wave 
equations (see \cite{Ai-Ifrim-Tataru-2020-global,Wu-2020-quartic} 
for recent results exploiting the characteristics of 2D equations). 
For the Muskat equation in subcritical spaces, 
there exist different ways to overcome the difficulties that appear in 3D. 
We refer to the recent article by Nguyen and Pausader~\cite{Nguyen-Pausader}, 
which studied the Muskat equation in a general framework, with an approach based on a 
paradifferential analysis of the Dirichlet-to-Neumann operator (\cite{ABZ-Invent}). 
However, the different approaches used so far seem, according to our 
understanding of the problem, 
inappropriate for the study of solutions with critical 
regularity. 
In this direction, the first breakthrough result in 3D was obtained by 
Gancedo and Lazar~\cite{Gancedo-Lazar-H2}, who 
worked with a new formulation of the 3D Muskat equation in terms 
of oscillatory integrals, and obtained 
the first result of global existence in
critical Sobolev space, for sufficiently small initial data. We will 
consider a different approach, which allows 
to study solutions with critical regularity and large data. In addition, 
we will obtain another proof of the main result of~\cite{Gancedo-Lazar-H2}. 

\subsection{The equation and its scaling invariance}
The Muskat equation is 
invariant by the transformation:
\be\label{acritical}
f(t,x)\mapsto f_\lambda(t,x)\defn\frac{1}{\lambda}f\left(\lambda t,\lambda x\right).
\ee
Then, by a direct calculation, one verifies that the spaces $\dot{W}^{1,\infty}(\xR^d)$ 
and $\dot{H}^{1+\frac{d}{2}}(\xR^{d})$ are critical spaces 
for the study of the Cauchy problem. This means that
$$
\lA \nabla f_\lambda\big\arrowvert_{t=0}
\rA_{L^{\infty}}=\lA \nabla f_0\rA_{L^{\infty}},\quad 
\lA f_\lambda\big\arrowvert_{t=0}
\rA_{\dot H^{1+\frac{d}{2}}}=\lA f_0\rA_{\dot H^{1+\frac{d}{2}}}.
$$
Given a real number $s\ge 0$, we denote by 
$H^s(\xR^2)$ (resp.\ $\dot{H}^{s}(\xR^2)$) the classical Sobolev (resp.\ homogeneous Sobolev) 
space of order $s$. 
They are equipped with the norms defined by
\be\label{norm:Sobolev}
\lA u\rA_{\dot{H}^{s}}^2=\int_{\xR^d} \la \xi\ra^{2s}\bla \hat{f}(\xi)\bra^2\dxi,
\quad 
\lA u\rA_{H^{s}}^2=\lA u\rA_{\dot{H}^s}^2+\lA u\rA_{L^2}^2.
\ee
In particular, when $d=2$, we have $\lA u\rA_{\dot{H}^{2}(\xR^2)}= (2\pi)^2\lA \Delta u\rA_{L^2(\xR^2)}$. 

\subsection{Cauchy problem}\label{S:1.2}

The study of the Cauchy problem for the Muskat equation 
with smooth enough initial data goes back to the works 
of Yi~(\cite{Yi2003}), Caflisch, Howison and Siegel 
 (\cite{SCH2004}), Ambrose~(\cite{Ambrose-2004,Ambrose-2007}) 
and C\'ordoba, C\'ordoba and Gancedo (\cite{CG-CMP,CCG-Annals}). 
This problem was extensively studied in the last decade. 
In particular, there are many local well-posedness results 
in sub-critical spaces: we refer to the works of  
Constantin, Gancedo, Shvydkoy and Vicol~\cite{CGSV-AIHP2017} 
for initial data in the Sobolev space 
$W^{2,p}(\xR)$ for some $p>1$, Cheng, Granero-Belinch\'on and 
Shkoller~\cite{Cheng-Belinchon-Shkoller-AdvMath} and Matioc~\cite{Matioc1,Matioc2} for initial data 
in $H^s(\xR)$ with $s > 3/2$ (see also~\cite{Alazard-Lazar,Nguyen-Pausader}), and Deng, Lei 
and Lin~\cite{DLL} for 2D initial data whose derivatives are H\"older continuous. 
The Muskat equation being parabolic, the proof of the local well-posedness 
results also provides global well-posedness results under a smallness assumption. 
The first global well-posedness results 
under mild smallness assumptions, namely assuming 
that the Lipschitz semi-norm is smaller than $1$, was 
obtained by Constantin, C{\'o}rdoba, Gancedo, Rodr{\'\i}guez-Piazza 
and Strain~\cite{CCGRPS-AJM2016} (see also \cite{CGSV-AIHP2017,PSt}).

In addition, there are three different 
cases where the Cauchy problem for solutions with critical regularity is well-understood: 

\begin{enumerate}[(i)]
\item \textbf{For Lipschitz initial data}, see  Camer\'on~\cite{Cameron,Cameron2020global,Cameron2020eventual}. 

\item \textbf{For small enough initial data}, see the results by C\'ordoba and Lazar~\cite{Cordoba-Lazar-H3/2} 
and Gancedo and Lazar~\cite{Gancedo-Lazar-H2} mentioned above.

\item \textbf{In one dimension}, see \cite{AN2,AN3} for the study of the 
Cauchy problem with arbitrary 
initial data in the one-dimensional critical Sobolev space $H^{3/2}(\xR)$. 
Compared to the Lipschitz or small data cases, we need to address the three following difficulties: 
$(i)$ the parabolic behavior degenerates at that level of regularity; $(ii)$ the lifespan of the solution 
depends on the initial data themselves and not only on its norm, and $(iii)$ the space $H^{3/2}(\xR)$ is not a Banach algebra (compared 
to $W^{1,\infty}(\xR)$). 
\end{enumerate}

In addition to these well-posedness results, let us mention that 
the existence and possible non-uniqueness 
of weak-solutions has also been thoroughly studied (we refer the reader to~\cite{Brenier2009,cordoba2011lack,szekelyhidi2012relaxation,castro2016mixing,forster2018piecewise,noisette2020mixing}). There are also blow-up results for certain large data; see the papers by 
Castro, C\'{o}rdoba, Fefferman, Gancedo and 
L\'opez-Fern\'andez~\cite{CCFG-ARMA-2013,CCFG-ARMA-2016,CCFGLF-Annals-2012}. 
They have proved that there are solutions such that 
at time $t=0$ the interface is a graph, at a subsequent time $t_1>0$ 
the interface is not a graph and then at a later time $t_2>t_1$, 
the interface is $C^3$ but not $C^4$. This result explains why 
it is interesting to prove the existence of solutions whose slopes 
can be arbitrarily large (as in~\cite{DLL,Cameron,Cordoba-Lazar-H3/2,Gancedo-Lazar-H2}) or even infinite 
(as we proved in~\cite{AN1,AN2,AN3}). 
Another consequence of this blow-up result is that 
one cannot prove a well-posedness result for large 
data such that the lifespan would depend only on the norm of the initial data (otherwise one would get a global existence result for any initial data 
by a scaling argument). This means that the lifespan necessarily depends on the profile of the initial data. 
This is a well-known problem in the analysis of several 
dispersive equations, or other parabolic equations (see in particular \cite{KNV-2007,Caffarelli-Vasseur-AoM,CV-2012,Silvestre-2012,VaVi,NgYa}). Here the problem is more difficult since the Muskat equation is fully nonlinear.

We will prove two different results about the Cauchy problem. 
Our main result is a local well-posedness result 
for arbitrary initial data in the critical space $W^{1,\infty}(\xR^2)\cap \dot{H}^{2}(\xR^2)$. 
An important new point is that we do not require a control of the low frequencies in $L^2(\xR^2)$. 
Namely, we will prove the following 
\begin{theorem}\label{main}
For any initial data $f_0$ in $W^{1,\infty}(\xR^2)\cap \dot H^{2}(\xR^2)$, 
there exists a time $T>0$ such that 
the Cauchy problem for the Muskat equation has a unique solution
\begin{equation*}
f\in L^\infty\big([0,T];W^{1,\infty}(\xR^2)\cap \dot H^{2}(\xR^2)\big) \cap L^2(0,T;\dot H^\frac52(\xR^2)).
\end{equation*}
\end{theorem}

The proof of Theorem~\ref{main} 
will immediately give an alternative proof to the following result, first proved in~\cite{Gancedo-Lazar-H2}. 
\begin{theorem}\label{main-2}
There exists a positive constant $\delta$ such that, 
for any initial data $f_0$ in $W^{1,\infty}(\xR)\cap  \dot H^{2}(\xR)$ satisfying
\begin{equation*}
\big(1+\lA \nabla f_0\rA_{L^{\infty}}^3\big)\lA f_0\rA_{\dot H^{2}}\leq \delta,
\end{equation*}
the Cauchy problem for the Muskat equation has a unique global solution
\begin{equation*}
f\in L^\infty\big([0,+\infty); W^{1,\infty}(\xR^2)\cap \dot H^{2}(\xR^2)\big)
\cap L^2(0,+\infty;\dot H^\frac52(\xR^2)).
\end{equation*}
\end{theorem}

As another consequence of the proof of Theorem~\ref{main}, 
we will obtain  the existence of solutions whose initial slope is infinite, which extends the 
main result in~\cite{AN1}
to the 3D Muskat problem.

\begin{theorem}\label{theo:main3}
Let $a=3/8$. 

$i)$ For any 
initial data $f_0$ in the following space
$$
\dot{H}^{2,\log^a}(\xR^2)
\defn\Big\{ u:\,\lA u\rA_{\dot{H}^{2,\log^a}}=\lA u\rA_{L^\infty}+ \blA |\xi|^2\big(\log(2+\la \xi\ra)\big)^{a}
 \hat{u}(\xi)\brA_{L^2}<+\infty\Big\},
$$

there exists a positive time $T$ such that the Cauchy problem for the Muskat 
equation~\e{Muskat} with initial data 
$f_0$ has 
a unique solution
$$
f\in C^0\big([0,T];\dot{H}^{2,\log^a}(\xR^2)\big)
\cap L^2\big(0,T;\dot{H}^{2}(\xR^2)\big).
$$ 
$ii)$ There exists a positive constant $c_0$ such that, if $f_0\in L^\infty(\mathbb{R}^2)$ and
$$
\int_{\xR^2} |\xi|^4\big(\log(2+\la \xi\ra)\big)^{2a}
\bla\hat{f}_0(\xi)\bra^2\dxi\le c_0,
$$
then the solution exists globally in time.
\end{theorem}
\subsection{Strategy of the proof and plan of the paper}

The proof is based on two different tools. On the one hand, 
we shall use the weighted fractional laplacians already used in 
our previous works~\cite{AN1,AN2,AN3}. We recall the needed results in 
Section~\ref{S:2}. 
On the other hand, we will introduce a decomposition of the nonlinearity 
into several terms playing different roles in Section~\ref{S:3}. 
On the technical side, we simplify many arguments compared to previous works 
on the subject by introducing an approach which is best carried out by the 
Fourier transform (instead of using Besov spaces or Triebel spaces).

\subsection{Notations}
Let us fix some basic notations.
\begin{enumerate}
\item The euclidean norm of $h=(h_1,h_2)\in \xR^2$ is 
denoted by $\la h\ra=\sqrt{h_1^2+h_2^2}$.
\item Given a scalar $a\in \xR$ or a vector 
$h\in \xR^2$, the Bracket notations reads:
$$
\langle a\rangle=\sqrt{1+a^2}, \quad \langle h\rangle =\sqrt{1+|h|^2}=\sqrt{1+h_1^2+h_2^2}.
$$
\item Given a non-zero vector $a$ in $\xR^2$, we set
$$
\check{a}=\frac{a}{\la a\ra}\cdot
$$
\item The operators $\delta_\a$ and $\Delta_\a$ are defined by
\begin{equation*}
\delta_\a f(x)=f(x)-f(x-\a),\quad \Delta_{\alpha}f(x)=\frac{f(x)-f(x-\alpha)}{|\alpha|}\cdot
\end{equation*}

\item If $A, B$ are nonnegative quantities, the inequality 
$A \lesssim B$ means that 
$A\le CB$ for some constant $C$ depending only on fixed quantities, and $A\sim B$ means that $A\les B\les A$.

\item Given two operators $A$ and $B$, the commutator $[A,B]$ is the difference 
$A\circ B-B\circ A$ .
\item Given a normed space~$X$
and a function~$\varphi=\varphi(t,x)$ defined on~$[0,T]\times \xR$ with values in~$X$,~$\varphi(t)$ the function~$x\mapsto
\varphi(t,x)$. In the same vein, we use $\lA \varphi\rA_X$ as 
a compact notation for the time dependent function 
$t\mapsto \lA \varphi(t)\rA_X$. 
\end{enumerate}

\section{Preliminaries}\label{S:2}

We gather in this section some inequalities that are used systematically in the sequel. 
We begin by summarizing some well-known estimates. 
Then we recall from our previous papers several estimates for the weighted 
fractional Laplacians $\D^{s,\phi}$. In the last paragraph, we introduce 
some estimates which allow to considerably simplify previous analysis of the Muskat equation. 

\subsection{Minkowski's inequality}
Consider two $\sigma$-finite measure spaces $(S_1,\mu_1)$ and $(S_2,\mu_2)$ 
and a measurable function $F\colon S_1\times S_2\to\xR$. 
Then, for all $p\in[1,+\infty)$,
\be\label{Min}
\bigg(\underset{S_2}{\int} 
\bigg\vert\underset{S_1}{\int} F(x,y)\mu_1(\dx)\bigg\vert^p 
\mu_2(\dy)\bigg)^{\frac{1}{p}}\le 
\bigg(\underset{S_1}{\int} 
\bigg(\underset{S_2}{\int} \la F(x,y)\ra^p \mu_2(\dy)\bigg)
^{\frac{1}{p}}
\mu_1(\dx).
\ee

\subsection{Sobolev embeddings}\label{S:2.2.a}
We will make extensive use of the Sobolev and Hardy-Littlewood-Sobolev inequalities. 
Recall that, in dimension two, Sobolev's inequality reads
\be\label{Sobolev}
\dot{H}^t(\xR^2)\hookrightarrow L^{\frac{2}{1-t}}(\xR^2) \quad\text{for}\quad 
0\le t<1.
\ee
Also, for for $s\in (0,1)$, there is a constant $c(s)$ such that
\be\label{Gagliardo}
\lA u\rA_{\dot H^{s}}^2=c(s) \underset{\xR^2\times \xR^2}{\iint}
\frac{\la u(x)-u(y)\ra^2}{\la x-y\ra^{2s}}\frac{\dx\dy}{\la x-y\ra^2}\cdot
\ee

In space dimension two, the Riesz potentials are defined by
$$
\mathbf{I}_s f(x)=\int_{\xR^2}\frac{f(y)}{|x-y|^{2-s}}\dy\quad\text{for } 0<s<2.
$$
Consider three positive numbers $(p,q,s)$ such that
$$
\frac{1}{p}-\frac{1}{q}=\frac{s}{2},\quad 1<p<\frac{2}{s}\cdot
$$
Then the Hardy-Littlewood-Sobolev inequality states that 
there exists a constant $C=C(p,q,s)$ such that, for all $f$ in 
$C^1_0(\xR^2)$,
\be\label{Riesz}
\lA \mathbf{I}_s f\rA_{L^q}\leqslant C\lA f\rA_{L^p}.
\ee
We will also make extensive use of the fact 
that, for $0<s<1$, the Fourier multiplier $\D^{-s}$ can be written as 
\be\label{b1}
\D^{-s}=c_{s}\mathbf{I}_{s},
\ee
for some constant $c_{s}$.

\subsection{Weighted fractional Laplacians}\label{S:2.3}

In our previous works~\cite{AN1,AN3} we introduced weighted fractional Laplacians and use them to study the 
Muskat equation in critical spaces. In this section, we extend the previous operators in dimension $d=2$. 

\begin{notation}
Consider $s\in [0,+\infty)$ and 
a function $\phi\colon [0,+\infty)\to [1,\infty)$. 
By definition, the weighted fractional Laplacian $\D^{s,\phi}$ denotes 
the Fourier multiplier with symbol $|\xi|^s\phi(\la\xi\ra)$, such that
\begin{equation*}
\mathcal{F}( \D^{s,\phi}f)(\xi)=|\xi|^s\phi(\la\xi\ra) \mathcal{F}(f)(\xi).
\end{equation*}
In addition, we define the space
$$
{H}^{s,\phi}(\xR)=\{ f\in L^2(\xR)\, :\,\D^{s,\phi}f\in L^2(\xR)\},
$$
equipped with the norm
$$
\lA f\rA_{{H}^{s,\phi}}\defn \lA f\rA_{L^2}
+\left(\int_\xR \la \xi\ra^{2s}(\phi(\la\xi\ra))^2\bla\hat{f}(\xi)\bra^2\dxi\right)^\mez.
$$
\end{notation}

\begin{remark}
There is a key difference between the space ${H}^{1,\phi}(\xR^2)$ 
and the space $H^{1+\eps}(\xR^2)$ with $\eps>0$. Indeed, the 
space ${H}^{1,\phi}(\xR^d)$ is not stable by 
product in general. 
\end{remark}
\begin{remark}
In the special case where $\phi(r)=\log(2+r)^a$, these operators were 
introduced and studied in~\cite{BN18a,BN18b,BN18d} for $s\in [0,1)$ (see also \cite{Ng}). 
\end{remark}

We shall consider special functions $\phi$ depending 
on a function $\kappa\colon [0,\infty)\to [1,\infty)$, of the 
form\footnote{The fact that this integral is well-defined follows at once from the assumptions on $\kappa$ below. 
The reason for this special choice is to obtain the identity~\e{dsphi} which connects $\D^{3/2,\phi}$ to an expression involving finite differences.} 
\begin{equation}\label{n10}
\phi(\lambda)=4\pi\int_0^\infty \frac{1-\cos(r)}{r^{3/2}}\kappa\left(\frac{\lam}{r}\right)\frac{\dr}{r}\cdot
\end{equation}
In addition we will always assume that $\kappa$ is an admissible weight, in the sense of the following definition.

\begin{definition}\label{defi:admissible}
An admissible weight is a function $\kappa\colon[0,\infty) \to [1,\infty)$ 
satisfying the following three conditions: 
\begin{enumerate}[$({{\rm H}}1)$]
\item\label{H1} $\kappa$ is increasing;
\item\label{H2} there exists a positive constant $c_0$ such that $\kappa(2r)\leq c_0\kappa(r)$ for any $r\geq 0$; 
\item\label{H3} the function $r\mapsto \kappa(r)/\log(4+r)$ is decreasing on  $[0,\infty)$.
\end{enumerate}
\end{definition}

The next propositions contains the main results about these operators. 
Their proofs are postponed to the appendix.

For later applications, we make some preliminary remarks about admissible weights.
\begin{lemma}\label{L:2.6}
For all $\sigma>0$, there exists $C_\sigma>0$ such that, for all  
$0< r\le \mu$,
\begin{align}
&r^\sigma\kappa\left(\frac{1}{r}\right) \le C_\sigma\,\mu^\sigma
\kappa\left(\frac{1}{\mu}\right),\label{est:kappa}\\
&r^\sigma\kappa^2\left(\frac{1}{r}\right) \le C_\sigma\,\mu^\sigma
\kappa^2\left(\frac{1}{\mu}\right).\label{est:kappatwo}
\end{align}
\end{lemma}
\begin{remark}
These inequalities have the following interpretation: even if the function $r\to \kappa(1/r)$ and $r\to \kappa^2(1/r)$ are 
decreasing, since the function $\kappa(r)/\log(2+r)$ is decreasing, one expects that 
$r^\sigma \kappa(1/r)$ and $r^\sigma \kappa^2(1/r)$ 
behave as increasing functions of $r$.
\end{remark}

We will see that  $\kappa$ and $\phi$ are equivalent (i.e. $\phi\sim \kappa$, see~\e{n60b}). 
The reason for introducing two different functions to code a single operator is that we will use them for different purposes. 
Indeed, it is convenient to use $\phi$ when we prefer to work with the 
frequency variable, whereas we will use $\kappa$ when 
the physical variable is more practical. The special choice~\e{n10} for the 
formula relating $\phi$ and $\kappa$ will be used to deduce the following identity~\e{dsphi} 
which allows to switch calculations between 
frequency and physical variables.

\begin{proposition}\label{Z9}
Assume that $\phi$ is defined by~\e{n10} for some admissible weight 
$\kappa$.  For all 
$g\in H^\infty(\xR^2)$, there holds
\begin{equation}\label{dsphi}
\D^{\tdm,\phi}g(x)=c\int_{\xR^2}\frac{2g(x)-g(x+\alpha)-g(x-\alpha)}{\la \alpha\ra^{3/2}}
\kappa\left(\frac{1}{|\alpha|}\right)\frac{\dalpha}{|\a|^2}\cdot
\end{equation}
\end{proposition}

Eventually, we will need  the following link between $\D^{s,\phi}$ and the 
function $\kappa$.

\begin{proposition}\label{P:3.5}
$i)$ Assume that $\phi$ is defined by~\e{n10} for some admissible weight 
$\kappa$. Define, for $g\in H^\infty(\xR^2)$, the semi-norm
$$
\lA g\rA_{s,\kappa}:=\left(\iint_{\xR^2\times\xR^2}
\la 2g(x)-g(x+\alpha)-g(x-\alpha)\ra^2
\left(\frac{1}{|\alpha|^{s}}\kappa\left(\frac{1}{|\alpha|}\right)\right)^2\frac{\dx \dalpha}{|\alpha|^2}\right)^\mez\cdot
$$
Then, for all $0< s<2$, there exist $c,C>0$ such that, for 
all $g\in H^\infty(\xR^2)$,
\begin{equation*}
c\int_{\xR}  \bla \D^{s,\phi}g(x)\bra^2\dx\le \lA g\rA_{s,\kappa}^2
\le C\int_{\xR}\bla \D^{s,\phi}g(x)\bra^2\dx.
\end{equation*}

$ii)$ There exist two constants $c,C>0$ such that, for all $\xi\in \xR^2$,
\be\label{n60b}
c\kappa(|\xi|)\le \phi(|\xi|)\le C \kappa(|\xi|).
\ee

\end{proposition}
\begin{proof}See Appendix~\ref{A:2}.
\end{proof}
\subsection{Estimates for the finite difference operators}

In the following lemma, we gather several estimates that will be widely used in the sequence.
\begin{lemma}\label{L:technical}
\begin{enumerate}[i)]
\item\label{tech:i} For all $a\in [0,+\infty)$ and $b\in (0,1)$, there exists $C>0$ such that
\be\label{tech:1}
\frac{1}{C}\lA f\rA_{\dot{H}^{a+b}}^2\le \int_{\xR^2} \Vert  \delta_{\alpha}f\Vert _{\dot H^{a}}^2
\frac{\dalpha}{|\alpha|^{2+2b}}\le C \lA f\rA_{\dot{H}^{a+b}}^2.
\ee
\item\label{item:Triebel} For all $b\in (0,1)$, 
there exists a constant $C$ such that
\be\label{tech:Triebel}
\bigg(\int_{\xR^2}\bigg(\int_{\xR^2}\la \delta_\a f(x)\ra^2\frac{\dalpha}{\la\a\ra^{2+2b}}\bigg)^2\dx\bigg)^\frac14 \le C \lA f\rA_{\dot{H}^{b+\frac{1}{2}}}.
\ee

\item \label{tech:ii} Assume that
\be\label{tech:assu}
a\in [0,+\infty),\quad \gamma\in [1,+\infty),\quad \gamma<b<2\gamma.
\ee
Then there exists $C>0$ such that
\be\label{tech:2}
\int_{\xR^2} \lA\delta_{\alpha} f-\a\cdot\nabla_x f\rA_{\dot H^{a}}^{2\gamma}\frac{\dalpha}{|\alpha|^{2+2b}}\le 
C\lA f\rA_{\dot{H}^{a+\frac{b}{\gamma}}}^{2\gamma}.
\ee
\item\label{tech:iii} Assume that $a,b,\gamma$ satisfy~\e{tech:assu}. Then there exists $C>0$ such that
\be\label{tech:3}
\int_{\xR^2} \lA\delta_{\alpha} f+\delta_{-\a}f\rA_{\dot H^{a}}^{2\gamma}\frac{\dalpha}{|\alpha|^{2+2b}}\le 
C\lA f\rA_{\dot{H}^{a+\frac{b}{\gamma}}}^{2\gamma}.
\ee
\end{enumerate}
\end{lemma}
\begin{proof}
$\ref{tech:i})$ Set $g=\D^a f$. Then
$$
\int_{\xR^2} \Vert  \delta_{\alpha}f\Vert _{\dot H^{a}}^2
\frac{\dalpha}{|\alpha|^{2+2b}}=\iint_{\xR^2} \frac{\la g(x)-g(x-\a)\ra^2}{|\alpha|^{2b}}
\frac{\dalpha}{|\alpha|^{2}}\cdot
$$
Consequently, \e{tech:1} is an immediate consequence of the equivalence 
of the homogeneous Sobolev norm~\e{norm:Sobolev} with the Gagliardo seminorm~\e{Gagliardo}.

$\ref{item:Triebel})$ It follows from Minkowski inequality~\e{Min} and the Sobolev embedding $\dot{H}^\frac12(\xR^2)\hookrightarrow L^4(\xR^2)$ that 
\begin{align*}
\bigg(\int_{\xR^2}\bigg(\int_{\xR^2}\la \delta_\a f(x)\ra^2\frac{\dalpha}{\la\a\ra^{2+2b}}\bigg)^2\dx\bigg)^\frac14
&\le\bigg( \int_{\xR^2} \blA  \delta_\a f \brA_{L^4}^2\frac{\dalpha}{|\alpha|^{2+2b}}\bigg)^{\mez}\\
&\le \bigg(\int_{\xR^2} \blA  \delta_\a f \brA_{\dot{H}^{\frac12}}^2\frac{\dalpha}{|\alpha|^{2+2b}}\bigg)^{\mez}
\end{align*}
Hence \e{tech:Triebel} follows from~\e{tech:1} applied with $a=1/4$. 

$\ref{tech:ii})$ The proof of this is best carried out by the Fourier transform. 
Indeed, 
$$
\mathcal{F}\left(\delta_{\alpha} f-\a\cdot\nabla_x f\right)(\xi)=
\left(1-e^{-i\a\cdot\xi}-i\a\cdot\xi\right)\hat{f}(\xi).
$$
In view of the elementary inequality
$$
\la 1-e^{-ia}-ia\ra\le \la a\ra\min\{1,\la a\ra\} \qquad (a\in \xR),
$$
we conclude that
\begin{multline*}
\int_{\xR^2} \lA\delta_{\alpha} f-\a\cdot\nabla_x f\rA_{\dot H^{a}}^{2\gamma}\frac{\dalpha}{|\alpha|^{2+2b}}\\
\le \int_{\xR^2} \bigg(\int \la\xi\ra^{2a} \big( \la \a\ra\la\xi\ra\min\{1,\la \a\ra\la\xi\ra\}\big)^2
\bla \hat{f}(\xi)\bra^2\dxi\bigg)^{\gamma}\frac{\dalpha}{|\alpha|^{2+2b}}.
\end{multline*}
As a result, it follows from Minkowski's inequality that
\begin{multline}\label{tech:n1}
\int_{\xR^2} \lA\delta_{\alpha} f-\a\cdot\nabla_x f\rA_{\dot H^{a}}^{2\gamma}\frac{\dalpha}{|\alpha|^{2+2b}}\\
\les \bigg(\int_{\xR^2} \bigg(\int_{\xR^2} \la \a\ra^{2\gamma}\la\xi\ra^{2\gamma}
\min\left\{1,\la \a\ra\la\xi\ra\right\}^{2\gamma}
\frac{\dalpha}{|\a|^{2+2b}}\bigg)^{\frac{1}{\gamma}}\la \xi\ra^{2a} \bla\hat{f}(\xi)\bra^2\dxi\bigg)^\gamma.
\end{multline}
Since $\gamma<b<2\gamma$ by assumption, we have 
$$
\int_{\xR^2} \la \a\ra^{2\gamma}\la\xi\ra^{2\gamma}\min\left\{1,\la \a\ra\la\xi\ra\right\}^{2\gamma}
\frac{\dalpha}{|\a|^{2+2b}}\les \la\xi\ra^{2b}.
$$
We thus obtain \e{tech:2} by reporting this inequality in \e{tech:n1}.

$\ref{tech:iii})$ By changing $\a$ into $-\a$, we deduce from~\e{tech:2} that
\be\label{tech:2-a}
\int_{\xR^2}\lA\delta_{-\alpha} f+\a\cdot\nabla_x f\rA_{\dot H^{a}}^{2\gamma}\frac{\dalpha}{|\alpha|^{2+2b}}\le 
C\lA f\rA_{\dot{H}^{a+\frac{b}{\gamma}}}^{2\gamma}.
\ee
Now, \e{tech:3} follows from \e{tech:2}, \e{tech:2-a} and the triangular inequality.
\end{proof}

\begin{lemma}Let $\gamma\in\{1,2\}$. 
\begin{enumerate}[i)]
\item\label{itemtech4:ii} 
For all $a\in [0,+\infty)$, and all $b,c\in (0,\gamma)$, there exists $C>0$ such that
\be\label{tech:5}
\iint_{\xR^2\times\xR^2} \Vert \delta_\alpha \delta_{h}f\Vert _{\dot H^{a}}^{2\gamma}
\kappa^{2\gamma}\left(\frac{1}{|h|}\right)\frac{\dh}{|h|^{2+2b}}\frac{\dalpha}{|\alpha|^{2+2c}}
\le C \blA \D^{a+\frac{b+c}{\gamma},\phi}f\brA_{L^2}^2.
\ee
\item\label{itemtech4:iii} For all $a\in [0,+\infty)$ and all $b\in (0,\gamma)$, there exists $C>0$ such that
\be\label{tech:4}
\int_{\xR^2} \Vert \delta_{h}f\Vert _{\dot H^{a}}^{2\gamma}
\kappa^{2\gamma}\left(\frac{1}{|h|}\right)\frac{\dh}{|h|^{2+2b}}
\le C \blA \D^{a+\frac{b}{\gamma},\phi}f\brA_{L^2}^2.
\ee
\end{enumerate}
\end{lemma}
\begin{proof}
Notice that \e{tech:5} follows immediately from \e{tech:1} and~\e{tech:4}. So it suffices to prove the latter estimate. 

By repeating arguments similar to the ones used in the previous proof, we begin by writing that, since
$$
\bla \widehat{\delta_{h} f}(\xi)\bra=\bla 1-e^{-i h\cdot\xi}\bra\bla\hat{f}(\xi)\bra\le
\min\{1,\la h\ra\la\xi\ra\}\bra\bla\hat{f}(\xi)\bra,
$$
we have
\begin{multline*}
\int_{\xR^2} \Vert \delta_{h}f\Vert _{\dot H^{a}}^{2\gamma}
\kappa^{2\gamma}\left(\frac{1}{|h|}\right)\frac{\dh}{|h|^{2+2b}}
\\
\le \int_{\xR^2} \bigg(\int_{\xR^2} \la\xi\ra^{2a} \min\{1,\la h\ra\la\xi\ra\}^2\bla \hat{f}(\xi)\bra^2\dxi\bigg)^{\gamma}
\kappa^{2\gamma}\left(\frac{1}{|h|}\right)\frac{\dh}{|h|^{2+2b}}.
\end{multline*}
Since $\gamma\ge 1$, one may apply Minkowski's inequality to infer that
\be\nonumber
\begin{aligned}
\int_{\xR^2} \Vert \delta_{h}f\Vert _{\dot H^{a}}^{2\gamma}
\kappa^{2\gamma}\left(\frac{1}{|h|}\right)\frac{\dh}{|h|^{2+2b}}
&\les \bigg(\int_{\xR^2}m(\xi)^{\frac{1}{\gamma}}
\la \xi\ra^{2a}\bla\hat{f}(\xi)\bra^2\dxi\bigg)^\gamma\quad\text{where}\\
m(\xi)&\defn \int_{\xR^2}\min\left\{1,\la h\ra\la\xi\ra\right\}^{2\gamma}
\kappa^{2\gamma}\left(\frac{1}{|h|}\right)\frac{\dh}{|h|^{2+2b}}\cdot
\end{aligned}
\ee
So, to obtain~\e{tech:4}, it will be sufficient to prove that $m(\xi)\les \phi^{2\gamma}(|\xi|)\la\xi\ra^{2b}$.
By Lemma \ref{L:2.6}, one has for $\varepsilon_0>0$
\begin{align*}
	(\min \{ 1,|\xi||h|\})^{\varepsilon_0} \kappa^{2\gamma}\left(\frac{1}{|h|}\right)\lesssim_{\varepsilon_0} \kappa^{2\gamma}(|\xi|)\sim  \phi^{2\gamma}(|\xi|).
\end{align*} 
It follows, for $\varepsilon_0>0$,
\begin{align*}
	m(\xi)\lesssim_{\varepsilon_0} \int_{\xR^2}\min\left\{1,\la h\ra\la\xi\ra\right\}^{2\gamma-\varepsilon_0}\frac{\dh}{|h|^{2+2b}}\phi^{2\gamma}(|\xi|)\sim_{\varepsilon_0}  \phi^{2\gamma}(|\xi|)\la\xi\ra^{2b}.
\end{align*}
This completes the proof.
\end{proof}

\section{Nonlinearity}\label{S:3}

Given two functions $f=f(x)$ and $g=g(x)$, introduce the notation
\be\label{defi:L(f)}
\Lr(f)g=-\frac{1}{2\pi}\int_{\mathbb{R}^2}
\frac{\alpha\cdot\nabla_x\Delta_\alpha g}{\langle \Delta_\alpha f\rangle^{3}}\frac{\dalpha}{|\alpha|^2},
\ee
where recall that
\begin{equation*}
\Delta_{\alpha}u(x)=\frac{u(x)-u(x-\alpha)}{|\alpha|},\quad 
\langle a\rangle=\sqrt{1+a^2},\quad \a\cdot\nabla_x=\alpha_1\partial_{x_1}+\a_2\partial_{x_2}.
\end{equation*}

With this notation, 
the Muskat equation reads
\be\label{Muskat2}
\partial_tf +\Lr(f)f=0.
\ee

Recall that  the linearized Muskat equation around the null solution reads 
$\partial_t f+\D f=0$ where $\D$ is fractional Laplacian $(-\Delta)^{1/2}$ defined by
\begin{equation*}
\mathcal{F}(\D u)(\xi)=|\xi| \mathcal{F}(u)(\xi).
\end{equation*}
Indeed, we have 
\begin{equation*}
\D u(x)=\frac{1}{2\pi}\int_{\mathbb{R}^2}\frac{u(x)-u(x-\alpha)}{|\alpha|}\frac{\dalpha}{|\a|^2},
\end{equation*}
and hence
\begin{equation*}
\D u(x)=-\frac{1}{2\pi}\int_{\mathbb{R}^2}\alpha\cdot\nabla_x\Delta_\alpha u(x)\frac{\dalpha}{|\alpha|^2},
\end{equation*}
as can be verified by integrating by parts in $\a$ (see the identity \e{A2} below applied with $\zeta=0$.) 

So, when $f=0$, we have $\Lr(0)=\D$. The next proposition gives a key decomposition of 
the general operator $\Lr(f)$ into three components:
$$
\Lr(f)g=
P(f)g+V(f)\cdot\nabla_x g+R(f,g),
$$
where $P(f)$ is an elliptic operator of order $1$, $V(f)$ is a vector field 
and $R(f,g)$ is a remainder term.

\begin{proposition}[Quasilinearization formula]\label{P:3.1}
There holds
\be\label{m5}
\Lr(f)g=
P(f)g+V(f)\cdot\nabla_x g+R(f,g)
\ee
where the operator $P(f)$ is defined by
\be\label{defi:P}
P(f)g
=\frac{1}{2\pi}\int_{\xR^2}
\frac{\delta_\a g}{\langle \ca\cdot\nabla_x  f\rangle^3} \frac{\dalpha}{|\a|^3}\quad\text{with}\quad\ca=\frac{\a}{|\a|}\in \xS^1,
\ee
and where the vector field $V(f)$ is given by
\begin{equation}\label{defi:V(f)}
V(f)(x)= \frac{1}{2\pi}\int_{\xR^2}\frac{1}{2}
\left(\frac{1}{\langle\Delta_{-\alpha} f(x)\rangle^{3}}-\frac{1}{\langle \Delta_{\alpha} f(x)\rangle^{3}}\right)
\alpha\frac{\dalpha}{|\alpha|^{3}}\in \mathbb{R}^2,
\end{equation}
and the remainder term $R(f,g)$ as the form 
\be\label{m7}
R(f,g)(x)=\frac{1}{2\pi}\int_{\xR^2}  M_\a(x)\delta_{\alpha}g(x)\frac{\dalpha}{|\a|^3},
\ee
for some symbol $M_\a$ satisfying the following pointwise bound:
\be\label{m7est}
\la M_\a(x)\ra \le   
6\la \Delta_\a f(x)- \ca\cdot\nabla_xf(x)\ra +3 \la \nabla_x(\delta_\a f)\ra.
\ee
\end{proposition}
\begin{proof}
The proof is in two steps. We begin by decomposing the 
nonlinearity into several pieces to obtain~\e{m5} with 
an explicit expression for the remainder term $R(f,g)$. 
Then, in the second step we will bound $R(f,g)$.

{\em Step 1: decompositions.} Introduce
\begin{align}
\mathcal{O}\left(\alpha,x\right)
&=\frac{1}{2}\frac{1}{\langle \Delta_{\alpha} f(x)\rangle^{3}}
-\frac{1}{2}\frac{1}{\langle \Delta_{-\alpha} f(x)\rangle^{3}},\label{Oss}\\
\mathcal{E}\left(\alpha,x\right) &=\frac{1}{2}\frac{1}{\langle \Delta_\alpha f(x)\rangle^{3}}
+\frac{1}{2}\frac{1}{\langle \Delta_{-\alpha} f(x)\rangle^{3}},\label{Ess}
\\
\mathcal{E}_0\left(\a,x\right)
&=\frac{1}{\langle \ca\cdot\nabla_x f (x)\rangle^{3}}\cdot\label{Esss}
\end{align}
The reasons to introduce these terms are the following. 
Firstly, one can decompose the coefficient 
$$
\frac{1}{\langle \Delta_\alpha f\rangle^{3}}
$$
into:
\be\label{m8}
\frac{1}{\langle \Delta_\alpha f\rangle^{3}}=\mathcal{E}\left(\a,\cdot\right)+\mathcal{O}\left(\alpha,\cdot\right).
\ee
Secondly, we have
$$
\mathcal{O}\left(-\alpha,x\right)=-\mathcal{O}\left(\alpha,x\right),\quad 
\mathcal{E}\left(-\alpha,x\right)=\mathcal{E}\left(\alpha,x\right).
$$
Eventually, we will replace $\mathcal{E}\left(\alpha,x\right) $ by $\mathcal{E}_0\left(\alpha,x\right)$ to the price 
of some error terms, involving $\Delta_{\alpha} f(x)-\ca\cdot\nabla f(x)$ or 
$\Delta_{-\alpha} f(x)+\ca\cdot\nabla f(x)$, which contribute to the remainder term.

Now, directly from the definition~\e{defi:L(f)} of $\Lr(f)g$ and~\e{m8}, we have 
\begin{align*}
\Lr(f)g = -\frac{1}{2\pi}\int_{\xR^2}\mathcal{E}\left(\alpha,\cdot\right) \a\cdot
\nabla_x\Delta_\alpha g(x)\frac{\dalpha}{|\alpha|^2}-\frac{1}{2\pi}\int_{\xR^2}\mathcal{O}\left(\alpha,\cdot\right)\a
\cdot \nabla_x\Delta_\alpha g(x)\frac{\dalpha}{|\alpha|^2}\cdot
\end{align*}
We further decompose the first term by writing
$$
\mathcal{E}\left(\alpha,\cdot\right)=\mathcal{E}_0\left(\a,\cdot\right)+\left(\mathcal{E}\left(\alpha,\cdot\right) -\mathcal{E}_0\left(\a,\cdot\right)\right),
$$ 
and, in the second term, we expand $\delta_\a g(x)$ into two parts: $g(x)$ and $-g(x-\a)$, that will be handle separately. 
It follows that 
\be\label{m10}
\Lr(f)g=
-\frac{1}{2\pi}\int_{\xR^2}\frac{1}{\langle \ca\cdot \nabla_x f\rangle^{3}}\a\cdot \nabla_x\Delta_\alpha g\frac{\dalpha}{|\alpha|^2}
+V(f)\cdot\nabla_x g+R(f,g),
\ee
where $V(f)$ is defined by~\e{defi:V(f)}, and 
$R(f,g)=R_1(f,g)+R_2(f,g)$ with
\begin{align}
R_1(f,g)&=
-\frac{1}{2\pi}\int_{\xR^2}\left(\mathcal{E}\left(\alpha,x\right)
-\mathcal{E}_0\left(\a,x\right)\right)\a\cdot \nabla_x\Delta_\alpha g(x)\frac{\dalpha}{|\alpha|^2}\label{defi:R1}\\
R_2(f,g)&=\frac{1}{2\pi}\int_{\xR^2}\mathcal{O}\left(\alpha,x\right) \a\cdot \nabla_x g(x-\alpha)\frac{\dalpha}{|\alpha|^{3}}\cdot\label{defi:R2}
\end{align}

The following lemma gives an alternative expression for the first term in the right-hand side of~\e{m10}, and 
thus establishes~\e{m10}.
\begin{lemma}\label{L:2.2}
For any vector 
$\zeta\in \mathbb{R}^2$, there holds
\begin{align}\label{A2}
-\int_{\xR^2} \frac{1}{\langle \ca\cdot \zeta\rangle^{3}} \a\cdot \nabla_x\Delta_\alpha g(x)\frac{\dalpha}{|\alpha|^2}
=\int_{\xR^2}\frac{1}{\langle \ca\cdot\zeta\rangle^3}  \frac{g(x)-g(x-\alpha)}{|\a|} \frac{\dalpha}{|\a|^2}\cdot
\end{align}
\end{lemma}
\begin{proof}
By symmetry, one has
\be\label{m13}
\int_{\xR^2}\frac{1}{\langle \zeta\cdot\ca\rangle^{3}} \frac{\a}{|\a|}\cdot \nabla_x g(x)\frac{\dalpha}{|\alpha|^2}=0.
\ee
On the other hand,
\be\label{m14}
\nabla_x(g(x-\a))=(\nabla_xg)(x-\a)=\nabla_\a(g(x)-g(x-\a)).
\ee
Hence, by integrating by parts,
\begin{align*}
-\int \frac{1}{\langle \ca\cdot\zeta\rangle^{3}} \a\cdot \nabla_x\Delta_\alpha g(x)\frac{\dalpha}{|\alpha|^2}
&=\int \frac{\nabla_\alpha (g(x)-g(x-\alpha))}{\langle \ca\cdot\zeta\rangle^{3}}\cdot\alpha\frac{\dalpha}{|\alpha|^{3}}\\
&=\int (g(x)-g(x-\alpha)) \cnx_\alpha\left(\frac{-\alpha}{\left(|\alpha|^2+(\a\cdot\zeta)^2\right)^{\frac{3}{2}}}\right) \dalpha.
\end{align*}
Now, note that  
\begin{align*}
\cnx_\alpha\left(\frac{-\alpha}{\left(|\alpha|^2+(\alpha\cdot\zeta)^2\right)^{\frac{3}{2}}}\right)
=\frac{-2}{\left(|\alpha|^2+(\alpha\cdot\zeta)^2\right)^{\frac{3}{2}}}
+\frac{3\alpha\cdot (\alpha+(\alpha\cdot\zeta)\zeta)}{\left(|\alpha|^2+(\alpha\cdot\zeta)^2\right)^{\frac{5}{2}}}
=\frac{1}{\langle\ca\cdot\zeta\rangle}\frac{1}{|\a|^3},
\end{align*}
which completes the proof.
\end{proof}

{\em Step 2: estimate of the remainder term.} 
It remains to prove the estimate~\e{m7est}. 

Now, we are going to exploit some symmetry to compute the remainder $R(f,g)$.
We begin by applying arguments parallel to those used to prove Lemma~\ref{L:2.2}. 
Firstly, since 
$\mathcal{E}\left(\alpha,x\right) -\mathcal{E}_0\left(\alpha,x\right)$ is symmetric with  
respect to $\alpha\mapsto -\alpha$, 
we have the following cancellation (parallel to~\e{m13}):
$$
\int_{\xR^2}\left(\mathcal{E}\left(\alpha,x\right)
-\mathcal{E}_0\left(\alpha,x\right)\right)\nabla_x g(x)\cdot\alpha\frac{\dalpha}{|\alpha|^3}=0.
$$
By considering that $\nabla_x(g(x-\a))=\nabla_\a(\delta_\a g)$ (see~\e{m14}), we see 
that the remainder term $R_1(f,g)$ (see~\e{defi:R1}) satisfies
$$
R_1(f,g)=\frac{1}{2\pi}\int_{\xR^2}\left(\mathcal{E}\left(\alpha,x\right)
-\mathcal{E}_0\left(\a,x\right)\right)\frac{\a\cdot \nabla_\a(\delta_\a g(x))}{|\a|}\frac{\dalpha}{|\alpha|^2}\cdot
$$
Then, by integrating by parts in $\a$, 
\begin{align*}
R_1(f,g)&=\frac{1}{2\pi}\int_{\xR^2}  A_\a(x)\delta_{\alpha}g(x) \dalpha,\quad \text{where}\\
A_\a(x)&\defn -\cnx_\alpha\left(\frac{\alpha}{|\alpha|^3}\left(\mathcal{E}\left(\alpha,x\right)
-\mathcal{E}_0\left(\alpha,x\right)\right)\right).
\end{align*}
Similarly, by using again $\nabla_x(g(x-\a))=\nabla_\a(\delta_\a g)$ and an integration by parts, we have
\begin{align*}
R_2(f,g)&=\frac{1}{2\pi}\int_{\xR^2}  B_\a(x)\delta_{\alpha}g(x) \dalpha,\quad \text{where}\\
B_\a(x)&\defn -\cnx_\alpha\left(\frac{\alpha}{|\alpha|^3}\mathcal{O}\left(\alpha,x\right) \right).
\end{align*}
By adding $R_1$ and $R_2$, and remembering \e{m8}, we conclude that 
the remainder term $R(f,g)$ can be written as
\begin{align*}
R(f,g)(x)&=\frac{1}{2\pi}\int_{\xR^2}  M_\a(x)\delta_{\alpha}g(x) \dalpha,
\end{align*}
where $M_\a(x)$ is given by
$$
M_\a(x)\defn -\cnx_\alpha\left(\frac{\alpha}{|\alpha|^3}
\left(\frac{1}{\langle\Delta_\a f(x)\rangle^3} -\frac{1}{\langle \ca\cdot\nabla_xf(x)\rangle^3}\right)\right).
$$

By considering that $\a\cdot\nabla_\a u(\a)=\frac{\diff}{\diff\!\lambda}u(\lambda \a)\arrowvert_{\lambda=1}$, one 
easily verifies that
\begin{align*}
&\cnx_\alpha\left(\frac{\alpha}{|\alpha|^3}\right)=-\frac{1}{|\a|^3},\\
&\a\cdot\nabla_\alpha\bigg(\frac{1}{\langle \ca\cdot\nabla_xf(x)\rangle^3}\bigg)=0,\\
&\a\cdot\nabla_\a \bigg(\frac{1}{\langle\Delta_\a f(x)\rangle^3}\bigg)=3\frac{\Delta_\a f(x)}{\langle \Delta_\a f(x)\rangle^5}\big(\Delta_\a f(x)-\ca\cdot\nabla_x f(x-\a)\big).
\end{align*}
It follows that
\begin{align*}
M_\a(x)&=\frac{1}{|\a|^3}\left(\frac{1}{\langle\Delta_\a f(x)\rangle^3} -\frac{1}{\langle \ca\cdot\nabla_xf(x)\rangle^3}\right)\\
&\quad-\frac{3}{|\a|^3}\frac{\Delta_\a f(x)}{\langle \Delta_\a f(x)\rangle^5}\big(\Delta_\a f(x)-\ca\cdot\nabla_x f(x-\a)\big).
\end{align*}
Since
$$
\nabla_x f(x-\a)=\nabla_x f(x)-\nabla_x(\delta_\a f(x)),
$$
we end up with
\begin{align*}
M_\a(x)&=\frac{1}{|\a|^3}\left(\frac{1}{\langle\Delta_\a f(x)\rangle^3} -\frac{1}{\langle \ca\cdot\nabla_xf(x)\rangle^3}\right)\\
&\quad-\frac{3}{|\a|^3}\frac{\Delta_\a f(x)}{\langle \Delta_\a f(x)\rangle^5}\big(\Delta_\a f(x)-\ca\cdot\nabla_x f(x)\big)
-\frac{3}{|\a|^3}\frac{\Delta_\a f}{\langle \Delta_\a f(x)\rangle^5} \ca\cdot\nabla_x(\delta_\a f)\cdot
\end{align*}
Since
$$
\la \frac{1}{\langle r_1\rangle^3}-\frac{1}{\langle r_2\rangle^3}\ra\le 3\la r_1-r_2\ra,
$$
we obtain that
$$
\la M_\a(x)\ra \le   
\frac{6}{|\a|^3}\la \Delta_\a f(x)- \ca\cdot\nabla_xf(x)\ra +\frac{3}{|\a|^3}\la \nabla_x(\delta_\a f)\ra.
$$
Therefore \e{m7} is proven. This completes the proof.\end{proof}

\section{Nonlinear estimates}\label{S:4}

The goal of this section is to prove the following
\begin{theorem}\label{mainthm0}
There exist a positive constant $C$ such that, for all $f\in H^\infty(\xR^2)$, 
	\begin{align}\label{Z5}
		&\int_{\xR^2} \Lr(f)f \D^{4,\phi^2} f \dx\geq \frac{\blA\D^{\frac{5}{2},\phi}f\brA_{L^2}^2}{1+\lA f\rA_{\dot W^{1,\infty}}^3}\\&
		\quad-C\blA\D^{2,\phi}f\brA_{L^2}\big(1+\blA\D^{2,\phi}f\brA_{L^2}^2\big)\blA\D^{\frac{5}{2},\phi}f\brA_{L^2}^2 \left(\phi\left(\frac{\blA\D^{\frac{5}{2},\phi}f\brA_{L^2}}{\blA\D^{2,\phi}f\brA_{L^2}}\right)\right)^{-1}\nonumber.
	\end{align}
\end{theorem}

The proof of this theorem is quite long and is separated into several steps. 
We rewrite the left-hand side of \e{Z5} as
$$
\int_{\xR^2}\Big( \Lr(f)\big( \D^{\tdm,\phi}f\big) \D^{\frac52,\phi} f +\big[\D^{\tdm,\phi}, \Lr(f)\big]f  \D^{\frac52,\phi} f\Big)\dx.
$$
We estimate the commutator in \S\ref{S:4.4}. To estimate the contribution of $\Lr(f)\big( \D^{\tdm,\phi}f\big)$, we  
use Proposition~\ref{P:3.1} and estimate the terms successively in paragraphs \S\ref{S:4.1}--\ref{S:4.3}.

\subsection{The convective term}\label{S:4.1}
We begin by studying the contribution of the convective term $V(f)\cdot\nabla_x$ which appears in Proposition~\ref{P:3.1}. We want to estimate
$$
\la \langle V(f)\cdot\nabla_x \D^{\frac{3}{2},\phi}f,\D^{\frac{5}{2},\phi}f\rangle\ra.
$$
To do so, we introduce the Riesz transform 
$\mathcal{R}=(\mathcal{R}_1,\mathcal{R}_2)$ which  
is defined by
\begin{equation*}
\mathcal{F}(\mathcal{R}(h))(\xi)=i\frac{\xi}{|\xi|}\hat h(\xi).
\end{equation*}
With this operator, one has
\begin{equation}
\nabla_x h(x)=\mathcal{R}\D(h)(x),\qquad
\mathcal{R}^{\star}=-\mathcal{R},
\end{equation}
where $\mathcal{R}^{\star}$ denotes the adjoint 
with respect to the $L^2(\xR^2)$-scalar product. 
It follows that for $g=\D^{\frac{5}{2},\phi}f$
\begin{align*}
\langle V(f)\cdot\nabla \D^{\frac{3}{2},\phi}f,g\rangle
=\langle V(f)\cdot\mathcal{R} g,g\rangle
=- \frac{1}{2} \big\langle \left[\mathcal{R}, V(f)\right ] g,g\big\rangle,
\end{align*}
where in the last identity we use the notation $[A,B]=AB-BA$. 

Consequently, we have to estimate the commutator 
$\left[\mathcal{R}, V(f)\right ]\cdot\nabla$. This is the purpose of the following propositon.

\begin{proposition}\label{X1}
There exists a positive constant $C$ such that, for all functions $f$ and $g$ in $H^\infty(\xR^2)$, 
\be\label{X1est}
\lA \left[\mathcal{R}, V(f)\right ] g\rA_{L^2} 
\le C \left(\lA f\rA_{\dot H^{\frac{9}{4}}}+
\lA f\rA_{\dot H^{\frac{17}{8}}}^2\right)\blA g\brA_{\dot{H}^{-\frac{1}{4}}}.
\ee
\end{proposition}
\begin{proof}The proof is in two steps. Firstly, we estimate the commutator 
between the Riesz transform and the multiplication by a function. 
\begin{lemma}
There exists a positive constant $C$ such that, for all functions $g_1$ and $g_2$ in $H^\infty(\xR^2)$,
\begin{equation}\label{A3}
\Vert \left[\mathcal{R}_j, g_1\right ] \partial_{x_k}g_2\Vert _{L^2}\lesssim
\Vert g_1\Vert _{\dot H^{\frac{5}{4}}}\Vert g_2\Vert _{\dot H^{\frac{3}{4}}}.
\end{equation}
\end{lemma}
\begin{proof}We follow the proof of a commutator 
estimate in our previous paper, see~\cite[Estimate $(60)$]{AN1}. 
Recall that the Riesz transform can be written under the form
$$
\mathcal{R}_jf(x)=-\frac{1}{2\pi}\lim_{\eps\to 0}\int_{\xR^2\setminus B(x,\eps)}
\frac{(x_j-y_j)}{|x-y|^3}f(y)\dy.
$$
It follows that
$$
\la \left[\mathcal{R}_j, g_1\right ] \partial_{x_k}g_2(x)\ra
=\frac{1}{2\pi}\la \int_{\mathbb{R}^2}(g_1(x)-g_1(y))\frac{x_j-y_j}{|x-y|^{3}}\partial_{y_k} (g(x)-g(y)) \dy\ra.
$$
Let us integrate by parts, and then integrate in $x$, to obtain
\be\label{n4Riesz}
\begin{aligned}
&\int_{\xR^2} \la\left[\mathcal{R}, g_1\right ] \cdot \nabla_x g_2(x)\ra^2\dx\\
&\qquad\lesssim 
\int_{\xR^2}\left(\int_{\mathbb{R}^2}\la \nabla g_1(y)\ra\la g_2(x)-g_2(y)\ra\frac{\dy}{|x-y|^{2}}\right)^2\dx\\
&\qquad\quad+\int_{\xR^2}\left(\int_{\mathbb{R}^2}\la g_1(x)-g_1(y)\ra\la g_2(x)-g_2(y)\ra\frac{\dy}{|x-y|^{3}}\right)^2\dx.
\end{aligned}
\ee
We will prove that the right-hand side in the previous inequality is bounded by 
$\Vert g_1\Vert _{\dot H^{\frac{5}{4}}}^2\Vert g_2\Vert _{\dot H^{\frac{3}{4}}}^2$. 
We begin by applying 
Cauchy-Schwarz inequality to get that the second 
term in the right-hand side above is bounded by
$$
\bigg( \int_\xR \bigg(\int_\xR \frac{|g_1(x)-g_1(y)|^2}{|x-y|^{2+3/2}}\dy\bigg)^2\dx\bigg)^{\mez}
\bigg( \int_\xR \bigg(\int_\xR \frac{|g_2(x)-g_2(y)|^2}{|x-y|^{2+1/2}}\dy\bigg)^2\dx\bigg)^{\mez}.
$$
By Lemma~\ref{L:technical}
\begin{align*}
\bigg( \int_{\xR^2} \bigg(\int_{\xR^2} \frac{|g_1(x)-g_1(y)|^2}{|x-y|^{2+3/2}}\dy\bigg)^2\dx\bigg)^{\mez}
\les \lA g_1\rA_{\dot{H}^{\mez+\frac34}}^2=\lA g_1\rA_{\dot{H}^{\frac54}}^2,
\end{align*}
and 
\be\label{Riesz-comm-10}
\bigg( \int_{\xR^2} \bigg(\int_{\xR^2} \frac{|g_2(x)-g_2(y)|^2}{|x-y|^{2+1/2}}\dy\bigg)^2\dx\bigg)^{\mez}\les \lA g_2\rA_{\dot{H}^{\mez+\frac14}}^2=\lA g_2\rA_{\dot{H}^\frac34}^2.
\ee

We now have to prove similar estimates for 
the first term in the right-hand side of \e{n4Riesz}. 
First, we apply the Cauchy-Schwarz inequality to bound it by
$$
 \bigg(\int_{\xR^2}\bigg(\int_{\xR^2}
 \frac{|\nabla g_1(y)|^{2}}{|x-y|^{2-1/2}} \dy\bigg)^2 \dx\bigg)^{\frac{1}{2}}
\bigg( \int_{\xR^2}\bigg(\int_{\xR^2}\frac{|g_2(x)-g_2(y)|^2}{|x-y|^{2+1/2}}\dy\bigg)^2 \dx\bigg)^{\frac{1}{2}}.
$$
The second factor is bounded by means of~\e{Riesz-comm-10}. 
To estimate the first one, we use 
the Riesz potential (see~\S\ref{S:2.2.a}) to write
$$
 \bigg(\int_{\xR^2}\bigg(\int_{\xR^2}
 \frac{|\nabla g_1(y)|^{2}}{|x-y|^{2-1/2}} \dy\bigg)^2 \dx\bigg)^{\frac{1}{2}}\les 
 \blA \mathbf{I}_{\frac{1}{2}}\big(|\nabla g_1|^{2}\big)\brA_{L^2}^{2}.
$$
Now, by considering the Hardy-Littlewood-Sobolev estimate~\e{Riesz}, 
we find that
\begin{align*}
\blA \mathbf{I}_{\frac{1}{2}}\big(|\nabla g_1|^{2}\big)\brA_{L^2}^{2}\le 
\blA \nabla g_1\brA_{L^{\frac{8}{3}}}^{2}\les \blA \nabla g_1\brA_{\dot{H}^\frac14}^{2}=\blA g_1\brA_{\dot{H}^\frac54}^{2}.
\end{align*}This completes the proof of the lemma.
\end{proof}
In view of the previous lemma, it remains only to estimate 
$\Vert V(f)\Vert _{\dot H^\frac{5}{4}}$, which is equivalent to estimate the 
$\dot{H}^\frac14(\xR^2)$-norm of $\nabla_x V(f)$. 
For the sake of shortness, introduce the function $F\colon \xR\to\xR$ defined by
$$
F(\tau)=\frac{1}{\langle \tau\rangle^3}\cdot
$$
Then recall that,
$$
V(f)(x)= \frac{1}{4\pi}\int_{\xR^2}\( F(\Delta_{-\a}f)-F(\Delta_\a f)\)  \alpha\frac{\dalpha}{|\alpha|^{3}}.
$$
Consider an index $1\le j\le 2$, set $f_j=\partial_{x_j}f$ and observe that one has the following decomposition
\begin{align*}
&\partial_j \( F(\Delta_{-\a}f)-F(\Delta_\a f)\)\\
&\qquad\qquad=- F'(\ca\cdot \nabla_x f)\(\Delta_{-\a}f_j+\Delta_{\a}f_j\)\\
&\qquad\qquad\quad+\mez \( F'(\Delta_{-\a}f)+F'(\Delta_\a f)\)\(\Delta_{-\a}f_j-\Delta_{\a}f_j\)\\
&\qquad\qquad\quad+\mez \( F'(\Delta_{-\a}f)-F'(\Delta_\a f)+2F'(\ca\cdot\nabla_x f)\)\(\Delta_{-\a}f_j+\Delta_{\a}f_j\).
\end{align*}
Now, since the function $F'$ is odd, bounded and Lipschitz, one has for any couple of real numbers $\tau_1,\tau_2$,
\begin{align*}
&\la F'(\tau_1)+F'(\tau_2)\ra=\la F'(\tau_1)-F'(-\tau_2)\ra\le \sup F'' \la \tau_1-(-\tau_2)\ra\les  \la \tau_1+\tau_2\ra,\\
&\la F'(\tau_1)-F'(\tau_2)\ra\le  2 \sup F'\les 1.
\end{align*}
This implies that
\be\label{decomposition:V}
 \partial_j \( F(\Delta_{-\a}f)-F(\Delta_\a f)\) =- F'(\ca\cdot \nabla_x f)\big( \Delta_{-\a}f_j+\Delta_{\a}f_j\big) +R_1+R_2+R_3,
\ee
where
\begin{align*}
\la R_1(\a,\cdot)\ra&\le \la \Delta_{-\a}f+\Delta_{\a}f\ra \( \la \Delta_{-\a}f_j\ra+\la \Delta_{\a}f_j\ra\),\\
\la R_2(\a,\cdot)\ra&\le \la \Delta_{\a}f-\ca\cdot\nabla f\ra \( \la \Delta_{-\a}f_j\ra+\la \Delta_{\a}f_j\ra\),\\
\la R_3(\a,\cdot)\ra&\le \la \Delta_{-\a}f+\ca\cdot\nabla_x f\ra\( \la \Delta_{-\a}f_j\ra+\la \Delta_{\a}f_j\ra\).
\end{align*}
These three remainder terms will be treated in a similar way. 

The two tricks in \e{decomposition:V} 
are that: $(i)$ this is an identity and not an inequality and $(ii)$ 
we have factored out a coefficient 
$F'(\ca\cdot \nabla_x f)$ in front of the symmetric difference $\Delta_{-\a}f_j+\Delta_{\a}f_j$ 
which does not depend on the 
length $|\a|$. Remembering that $F'$ is bounded and using polar coordinates, 
we can exploit these two facts by writing
$$
\la \int_{\xR^2}F'(\ca\cdot \nabla_x f)\big( \Delta_{-\a}f_j+\Delta_{\a}f_j\big)\frac{\dalpha}{|\a|^2}\ra
\les 
\int_{\xS^1}\la \int_0^{+\infty} \big(\Delta_{-r\theta }f_j+\Delta_{r\theta}f_j\big)\frac{\dr}{r}\ra \diff \! \mathcal{H}^1.
$$
Observe that
\begin{multline*}
\lA \int_{\xR^2}F'(\ca\cdot \nabla_x f)\big( \Delta_{-\a}f_j+\Delta_{\a}f_j\big)\frac{\dalpha}{|\a|^2}\rA_{\dot H^\frac{5}{4}}\\
\les 
\int_{\xS^1}\lA \int_0^{+\infty} \big(\Delta_{-r\theta }f_j+\Delta_{r\theta}f_j\big)\frac{\dr}{r}\rA_{\dot{H}^\frac14} \diff \! \mathcal{H}^1(\theta).
\end{multline*}
The $\dot{H}^\frac{1}{4}$-norm of the inner integral is computed by means 
of the Plancherel's identity. 
Here we use again the fact that the symbol of the symmetric difference operator $\delta_{-\a}+\delta_\a$ is given by 
$2-2\cos(\a\cdot \xi)$. This implies that for any $\theta\in \mathbb{S}^1$, 
\begin{align*}
&\lA \int_0^{+\infty} \big(\Delta_{-r\theta }f_j+\Delta_{r\theta}f_j\big)\frac{\dr}{r}\rA_{\dot{H}^\frac14}^2\\
&\qquad\qquad\le \int_{\xR^2}\int_0^{+\infty}4(1-\cos( r\theta\cdot\xi))^2 \la \xi\ra^\mez \bla \xi_j\hat{f}(\xi)\bra^2\frac{\dxi\dr}{r^3}\\
&\qquad\qquad\les  \int_{\xR^2}\la\theta\cdot\xi\ra \la \xi\ra^\mez   \bla \xi_j\hat{f}(\xi)\bra^2 \dxi\le \lA f\rA_{\dot{H}^\frac94}^2.
\end{align*}
This proves that
$$
\la \int_{\xR^2}F'(\ca\cdot \nabla_x f)\big( \Delta_{-\a}f_j+\Delta_{\a}f_j\big)\frac{\dalpha}{|\a|^2}\ra
\les \lA f\rA_{\dot{H}^\frac94}^2,
$$
which completes the analysis of the contribution of the first term in the right-hand side of \e{decomposition:V}. 

It remains to estimate the contributions of the remainder terms. 
To do so, we use three elementary ingredients: $(i)$ a symmetry argument (using the change of variables $\a\mapsto -\a$ to 
reduce the number of terms that need to be estimated), $(ii)$ the 
fact that the product is continuous from 
$\dot H^\frac{5}{8}(\xR^2)\times \dot H^\frac{5}{8}(\xR^2)$ 
into $\dot H^\frac{1}{4}(\xR^2)$. It follows that
\begin{align*}
\sum_{\ell=1}^3\lA R_\ell\rA_{\dot H^\frac{1}{4}}
&\les \int \lA \Delta_{-\a}f+\Delta_{\a}f\rA_{\dot{H}^\frac58}\lA \Delta_\a f_j\rA_{\dot{H}^\frac58}\frac{\dalpha}{|\a|^2}\\
&\quad +\int \lA \Delta_{\a}f-\ca\cdot\nabla f\rA_{\dot{H}^\frac58}\big(\lA \Delta_\a f_j\rA_{\dot{H}^\frac58}+\lA \Delta_{-\a} f_j\rA_{\dot{H}^\frac58}\big)\frac{\dalpha}{|\a|^2}.
\end{align*}
Therefore
\begin{align*}
\sum_{\ell=1}^3\lA R_\ell\rA_{\dot H^\frac{1}{4}}
&\lesssim\left(\int_{\xR^2}
\Vert \delta_{\alpha} f-\a\cdot \nabla f\Vert _{\dot H^{\frac{5}{8}}}^2
\frac{\dalpha}{|\alpha|^{2+3}} \right)^{\frac{1}{2}}\left(\int_{\xR^2} \Vert \delta_{\alpha} \nabla f\Vert _{\dot H^{\frac{5}{8}}}^2
\frac{\dalpha}{|\alpha|^{2+1}}\right)^{\frac{1}{2}}\\
&\quad+\left(\int_{\xR^2}
\Vert \delta_{-\alpha} f+\delta_{\alpha} f\Vert _{\dot H^{\frac{5}{8}}}^2
\frac{\dalpha}{|\alpha|^{2+3}}\right)^{\frac{1}{2}}
\left(\int_{\xR^2} \Vert \delta_{\alpha} \nabla f\Vert _{\dot H^{\frac{5}{8}}}^2
\frac{\dalpha}{|\alpha|^{2+1}}\right)^{\frac{1}{2}}.
\end{align*}
Hence, we deduce from Lemma~\ref{L:technical} that
$$
\sum_{\ell=1}^3\lA R_\ell\rA_{\dot H^\frac{1}{4}}\lesssim \lA f\rA_{\dot H^{\frac{17}{8}}}^2.
$$
The desired result~\e{X1est}
then follows from~\e{A3}. This completes the proof.
\end{proof}

\subsection{The elliptic term}\label{S:4.2}
We now move to the analysis of the elliptic term.

\begin{lemma}\label{lem1}
There exist two a positive constant $C$ such that, for all 
functions $f$ and $g$ in $H^\infty(\xR^2)$, 
\begin{align*}
	\langle P(f)g,\D g\rangle &\geq \frac{ \lA \nabla g\rA_{L^2}^2}{\langle \lA f\rA_{\dot{W}^{1,\infty}}\rangle^3} 	-C\lA f\rA_{\dot H^{\frac{9}{4}}}\lA g\rA_{\dot{H}^{\frac{7}{8}}}^2.
\end{align*} 
\end{lemma}
\begin{proof}Define $\Phi(r)=1/\langle r\rangle^3$ and set $h=\D^{\mez}g$ so that 
$\D g=\D^{\frac{1}{2}}h$, $g=\D^{-\frac{1}{2}}h$. 
It follows from the definition of $P(f)$ (see~\e{defi:P}) 
that
\begin{align*}
\langle P(f)g,\D g\rangle
&= \frac{1}{2\pi}\underset{\xR^2\times\xR^2}{\iint}\Phi\left(\ca\cdot\nabla_x f(x)\right)
\D^{-\frac{1}{2}}\delta_\alpha h(x)\D^{\frac{1}{2}} h(x) \frac{\dalpha\dx}{|\alpha|^3}	\\
&= \frac{1}{4\pi}\underset{\xR^2\times\xR^2}{\iint}\Phi\left(\ca\cdot\nabla_x f(x)\right)
\D^{-\frac{1}{2}}\delta_\alpha h(x)\D^{\frac{1}{2}}\delta_\alpha h(x) \frac{\dalpha\dx}{|\alpha|^3},
\end{align*}
where, to make appear $\delta_\a h$ in the second line, as usual, we have split the integral $I$ at stake into $I=\mez I+\mez I$ 
and then made an elementary change of variable to replace the factor $\D^{\mez}h(x)$ by $(\D^\mez h)(x-\a)$ in the second half.  
Then, by using Plancherel's identity, it follows that
$$
\langle P(f)g,\D g\rangle= \frac{1}{4\pi}\underset{\xR^2\times\xR^2}{\iint} \D^{\frac{1}{2}}\left(\Phi\left(\ca\cdot\nabla_x f\right)
\delta_\alpha\D^{-\frac{1}{2}}h\right)(x) \delta_\alpha h(x)\frac{\dalpha\dx}{|\alpha|^3}\cdot
$$
Now we want to commute $\D^{\frac{1}{2}}$ with the multiplication by $\Phi\left(\ca\cdot\nabla_x f\right)$. 
To do so, recall that the operator $\D^{\frac{1}{2}}$ can be written as
$$
\D^\mez u=\frac{2^\mez \Gamma(5/4)}{\pi\la \Gamma(-1/4)\ra} \int_{\xR^2}
\frac{u(x)-u(x-z)}{\la z\ra^{1/2}}\frac{\dz}{|z|^2}.
$$
It follows that
\begin{equation*}
\bla \D^{\frac{1}{2}}(f_1 f_2)(x)-f_1(x)\D^{\frac{1}{2}}f_2(x)\bra
\lesssim\int_{\xR^2}\frac{\la f_1(x)-f_1(x-z)\ra\la f_2(x-z)\ra}{|z|^{1/2}}\frac{\dz}{|z|^2}\cdot
\end{equation*}	
Consequently, by setting
$$
I_0\defn \underset{\xR^2\times\xR^2}{\iint}
\Phi\left(\ca\cdot\nabla_x f(x)\right) \la h(x)-h(x-\alpha)\ra^2\frac{\dalpha\dx}{|\alpha|^3},
$$
we get that
$$
\la\langle P(f)g,\D g\rangle-I_0 \ra \lesssim \underset{(\xR^2)^3}{\iiint}  \la \delta_z \Phi\left(\ca\cdot\nabla_x f\right)\ra
\bla\delta_\alpha\D^{-\frac{1}{2}}h(x-z)\bra 
\la\delta_\alpha h(x)\ra   \frac{\dz\dalpha\dx}{|z|^{\frac{5}{2}}|\alpha|^3}\cdot
$$ 
Directly from the definitions of $\Phi$ and $g$, we notice that
$$
I_0\ge \underset{\xR^2\times\xR^2}{\iint} \frac{1}{\langle \nabla f(x)\rangle^3}
\left(\frac{\D^{\mez}g(x)-\D^{\mez}g(y)}{\la x-y\ra ^\mez}\right)^2\frac{\dx\dy}{|x-y|^2}.
$$
Consequently, it remains only to prove that
\be\label{claim:P(f)}
\la	\langle P(f)g,\D g\rangle-I_0 \ra 
\lesssim  \lA f\rA_{\dot H^{\frac{9}{4}}}\lA h\rA_{\dot H^{\frac{3}{8}}}^2.
\ee

To do so, we begin by using the elementary estimate 
$$
\la \delta_z \Phi\left(\ca\cdot\nabla_x f\right)\ra\les  \la\nabla_x \delta_zf\ra.
$$
By combining this with 
the Cauchy-Schwarz inequality, we obtain
\begin{multline*}
\la	\langle P(f)g,\D g\rangle-I_0 \ra \\ 
\lesssim  \iint \left(\int  | \nabla_x \delta_zf(x)|^2\frac{\dz}{|z|^{\frac{7}{2}}} \right)^{\frac{1}{2}}
\left(\int  \bla \delta_\alpha\D^{-\frac{1}{2}}h(x-z)\bra^2\frac{\dz}{|z|^{\frac{3}{2}}} \right)^{\frac{1}{2}}
|\delta_\alpha h(x)|\dx
\frac{\dalpha }{|\alpha|^3}\cdot
\end{multline*}
Using the H\"older inequality $L^4\cdot L^4\cdot L^2\subset L^2$, this gives
\begin{equation*}
\la	\langle P(f)g,\D g\rangle-I_0 \ra 
\lesssim \lA\nabla_x f\rA_{\dot{F}^{\frac{3}{4}}_{4,2}}
\int\blA\mathbf{I}_{\frac{1}{2}}\big(|\delta_\alpha\D^{-\frac{1}{2}}h|^2\big)\brA_{L^2}^{\frac{1}{2}}
\lA\delta_\alpha h\rA_{L^2}  \frac{\dalpha }{|\alpha|^3},
\end{equation*}
where we used the notation
$$
\lA\nabla_x f\rA_{\dot{F}^{\frac{3}{4}}_{4,2}}=
\bigg(\int_{\xR^2}\bigg(\int_{\xR^2}\la \delta_z \nabla_x f(x)\ra^2\frac{\dz}{\la z\ra^{2+2\frac{3}{4}}}\bigg)^2\dx\bigg)^\frac14.
$$
It follows from the estimate~\e{tech:Triebel} that
$$
\lA\nabla_x f\rA_{F^{\frac{3}{4}}_{4,2}}\les \lA f\rA_{\dot H^{\frac{9}{4}}}.
$$
On the other hand, the classical estimate for Riesz potentials (see~\e{Riesz}) implies that
$$
\blA\mathbf{I}_{\frac{1}{2}}\big(|\delta_\alpha\D^{-\frac{1}{2}}h|^2\big)\brA_{L^2}^{1/2}
\lesssim \brA\delta_\alpha\D^{-\frac{1}{2}}h\brA_{L^{\frac{8}{3}}}\lesssim \brA\delta_\alpha\D^{-\frac{1}{4}}h\brA_{L^{2}}.
$$
It follows that
\begin{align*}
&\la	\langle P(f)g,\D g\rangle-I_0 \ra \\
&\qquad\lesssim \lA f\rA_{\dot H^{\frac{9}{4}}} \int_{\xR^2}  \brA\delta_\alpha\D^{-\frac{1}{4}}h\brA_{L^{2}}
\lA\delta_\alpha h\rA_{L^2}  \frac{\dalpha }{|\alpha|^3}\\
&\qquad\lesssim \lA f\rA_{\dot H^{\frac{9}{4}}}
\left(\int _{\xR^2} \brA\delta_\alpha\D^{-\frac{1}{4}}h\brA_{L^{2}}^2 \frac{\dalpha }{|\alpha|^{2+\frac{7}{4}}}\right)^{\frac{1}{2}}
\left( \int _{\xR^2}  \lA\delta_\alpha h\rA_{L^2}^2  \frac{\dalpha }{|\alpha|^{2+\frac{3}{4}}}\right)^{\frac{1}{2}}\\
&\qquad\lesssim \lA f\rA_{\dot H^{\frac{9}{4}}}\lA h\rA_{\dot H^{\frac{3}{8}}}^2
\end{align*}
where we have used the elementary estimate~\e{tech:1} to obtain the last inequality. 
This implies the wanted estimate~\e{claim:P(f)} 
since $\lA h\rA_{\dot H^{\frac{3}{8}}}\les \lA g\rA_{\dot{H}^{\frac{7}{8}}}$. 
This completes the proof of the proposition.
\end{proof}

\subsection{The remainder term}\label{S:4.3}
It remains to estimate the remainder term $R(f,g)$ which is given Proposition~\ref{P:3.1}. 
\begin{proposition}\label{Z19}
There exists a positive constant $C$ such that, for all 
functions $f$ and $g$ in $H^\infty(\xR^2)$,
\begin{equation}\label{Z14}
\Vert R(f,g)\Vert _{L^2}\le C \lA f\rA_{\dot H^{\frac{9}{4}}}\lA g\rA_{\dot H^{\frac{3}{4}}}.
\end{equation}
In particular, there holds
\begin{equation}\label{Z15}
\blA R(f,\D^{\frac{3}{2},\phi}f)\brA_{L^2}\le C\lA f\rA_{\dot H^{\frac{9}{4}}}\Vert \D^{\frac{9}{4},\phi}f\Vert _{L^2}.
\end{equation}
\end{proposition}
\begin{proof}
Recall that 
the remainder term $R(f,g)$ has the form 
$$
R(f,g)(x)=\frac{1}{2\pi}\int_{\xR^2}  M_\a(x)\delta_{\alpha}g(x) \dalpha,
$$
for some symbol $M_\a$ satisfying the following point-wise bound:
$$
\la M_\a(x)\ra \le   
\frac{6}{|\a|^3}\la \Delta_\a f(x)- \ca\cdot\nabla_xf(x)\ra +\frac{3}{|\a|^3}\la \nabla_x(\delta_\a f)\ra.
$$
By using successively the 
Minskowski, H\"older, Cauchy-Schwarz 
and Sobolev inequalities, we deduce that
\begin{align*}
\lA R(f,g)\rA_{L^2}
&\lesssim   \int_{\xR^2} \frac{1}{|\a|}\Vert  M_\alpha\Vert_{L^4}\Vert\delta_{\alpha}g\Vert_{L^4} \frac{\dalpha}{|\a|^2}\\
&\lesssim  \bigg( \int_{\xR^2} \lA M_\a\rA_{L^4}^2 \frac{ \dalpha}{|\alpha|^{\frac{7}{2}}}\bigg)^{\frac{1}{2}} 
\bigg(\int_{\xR^2}\blA\delta_{\alpha}\D^\mez g\brA_{L^2}^2\frac{ \dalpha}{|\alpha|^{\frac{5}{2}}}\bigg)^{\frac{1}{2}}\\
&\lesssim  \bigg( \int_{\xR^2} \lA M_\a\rA_{L^4}^2 \frac{ \dalpha}{|\alpha|^{\frac{7}{2}}}\bigg)^{\frac{1}{2}}
\lA g\rA_{\dot H^{\frac{3}{4}}}
\end{align*}
where we have used~\e{Gagliardo} to obtain the last inequality. 

Now, by H\"older inequality, we have
\begin{align*}
\int_{\xR^2}  \lA M_\a\rA_{L^4}^2 \frac{ \dalpha}{|\alpha|^{\frac{7}{2}}}
&\lesssim \int_{\xR^2} \lA\nabla_x\delta_{\alpha}f\rA_{L^4}^2
+\lA \Delta_{\alpha} f-\ca\cdot \nabla_x f\rA_{L^4}^2
\frac{\dalpha}{|\alpha|^{\frac{7}{2}}}\\
&\lesssim   \int_{\xR^2} \lA\delta_{\alpha}f\rA_{\dot H^{\frac{3}{2}}}^2
+\lA\Delta_{\alpha} f-\ca\cdot\nabla_x f\rA_{\dot H^{\mez}}^2
\frac{\dalpha}{|\alpha|^{\frac{7}{2}}}
=c\lA f\rA_{\dot H^{\frac{9}{4}}}^2
\end{align*}
where the last equation is obtained by using Plancherel's identity. This completes the proof.
\end{proof}

\subsection{The commutator estimate}\label{S:4.4}

\begin{proposition}\label{Z18}
There exists a positive constant $C$ such that, for all 
function $f$ in $H^\infty(\xR^2)$,
\begin{align*}
&\blA \big[ \D^{\frac{3}{2},\phi},\mathcal{L}(f)\big](f)\brA_{L^2}\\
&\qquad\qquad\le C
\lA f\rA_{\dot H^{\frac{9}{4}}}\Big(\blA \D^{\frac{9}{4},\phi}f\brA_{L^2}\\
&\qquad\qquad\qquad\qquad\qquad+\lA f\rA_{\dot{H}^{\frac{17}{8}}}\blA \D^{\frac{17}{18},\phi}f\brA_{L^2}+
\lA f\rA_{\dot{H}^{\frac{25}{12}}}^2\blA \D^{\frac{25}{12},\phi}f\brA_{L^2}\Big).
\end{align*}
\end{proposition}
\begin{proof}

Recall that by definition
$$
\mathcal{L}(f)f
=-\frac{1}{2\pi}\int_{\mathbb{R}^2}F_\alpha(x)G_\a(x)\frac{\dalpha}{|\alpha|^2}
\quad\text{where}\quad
\left\{
\begin{aligned}
F_\alpha&=\frac{1}{\langle \Delta_\alpha f(x)\rangle^{3}}, \\
G  _\a&=\a\cdot\nabla_x\Delta_\alpha f(x).
\end{aligned}
\right.
$$
Let us introduce 
$$
\Gamma_\alpha\defn\D^{\frac{3}{2},\phi}\left[F_\alpha G_\alpha \right]-F_\alpha\D^{\frac{3}{2},\phi}
G_\a - G_\a\D^{\frac{3}{2},\phi}\left[F_\alpha\right].
$$
So, 
\begin{align*}
&\blA \big[ \D^{\frac{3}{2},\phi},\mathcal{L}(f)\big](f)\brA_{L^2} \lesssim (I)+(II) \quad\text{where},\\
&(I)\defn \bigg(\int\left(	\int G_\alpha(x)\D^{\frac{3}{2},\phi}F_\alpha(x)\frac{\dalpha}{|\alpha|^2}\right)^2\dx\bigg)^\mez,\\
&(II)\defn
\bigg(\int\left(\int \Gamma_\alpha(x)\frac{\dalpha}{|\alpha|^2}\right)^2\dx\bigg)^\mez.
\end{align*}
We will prove that 
\begin{align}
(I)&\lesssim 
\lA f\rA_{\dot H^{\frac{9}{4}}}\Big(\blA \D^{\frac{9}{4},\phi}f\brA_{L^2}\label{Z11}\\
&\qquad\qquad\qquad+\lA f\rA_{\dot{H}^{\frac{17}{8}}}\blA \D^{\frac{17}{18},\phi}f\brA_{L^2}+\blA f\brA_{\dot{H}^{\frac{25}{12}}}^2\blA \D^{\frac{25}{12},\phi}f\brA_{L^2}\Big),\notag\\
(II)&\lesssim  \lA f\rA_{\dot H^{\frac{9}{4}}}\blA \D^{\frac{9}{4},\phi}f\brA_{L^2}.\label{Z13}
\end{align}

\textit{Step 1:} We prove \eqref{Z11}. 
Starting from Minkowski's inequality, write
$$
(I)=\lA \int G_\alpha \D^{\frac{3}{2},\phi}F_\alpha \frac{\dalpha}{|\alpha|^2}\rA_{L^2}
\le \int \lA G_\alpha \D^{\frac{3}{2},\phi}F_\alpha  \rA_{L^2} \frac{\dalpha}{|\alpha|^2}.
$$
Since the product is continuous from $\dot{H}^{\mez}(\xR^2)\times \dot{H}^{\mez}(\xR^2)$ to $L^2(\xR^2)$ 
(in view of the H\"older inequality $L^4\cdot L^4\hookrightarrow L^2$ and the Sobolev embedding 
$\dot{H}^{\mez}(\xR^2)\hookrightarrow L^4(\xR^2)$), this gives
\begin{align*}
(I)&\les \int \lA G_\alpha\rA_{\dot{H}^{\mez}} \blA \D^{\frac{3}{2},\phi}F_\alpha \brA_{\dot{H}^{\mez}} \frac{\dalpha}{|\alpha|^2}\\
&\les \bigg( \int \lA G_\alpha\rA_{\dot{H}^{\mez}}^2\frac{\dalpha}{|\alpha|^{7/2}}\bigg)^\mez
\bigg( \int \blA \D^{\frac{3}{2},\phi}F_\alpha \brA_{\dot{H}^{\mez}}^2 \frac{\dalpha}{|\alpha|^{1/2}}\bigg)^\mez.
\end{align*}
The first factor is estimated directly from the definition of $G_\a$. Indeed, using Plancherel's identity, one has
\begin{align*}
\int \lA G_\alpha\rA_{\dot{H}^{\mez}}^2\frac{\dalpha}{|\alpha|^{7/2}}= \int |\xi|\left|\alpha.\xi \frac{1-e^{i\alpha.\xi}}{|\alpha|}\right|^2 \frac{\dalpha}{|\alpha|^{7/2}}|\hat f(\xi)|^2d\xi= c\lA f\rA_{\dot H^{\frac{9}{4}}}^2. 
\end{align*}
The analysis of the second factor is more difficult. 
Set
$$
Q= \int \blA \D^{\frac{3}{2},\phi}F_\alpha \brA_{\dot{H}^{\mez}}^2 \frac{\dalpha}{|\alpha|^{1/2}}.
$$
We claim that $Q$ satisfies 
$$
Q\les \blA \D^{2,\phi}f\brA_{\dot{H}^{\frac{1}{4}}}^2+\lA f\rA_{\dot{H}^{\frac{17}{8}}}^2\blA \D^{\frac{17}{18},\phi}f\brA_{L^2}^2+ \blA f\brA_{\dot{H}^{\frac{25}{12}}}^4\blA \D^{\frac{25}{12},\phi}f\brA_{L^2}^2.
$$
This will imply the wanted result~\e{Z11}. 

Notice that 
$$
\blA \D^{\frac{3}{2},\phi}F_\alpha \brA_{\dot{H}^{\mez}}^2=\blA \D^{2,\phi}F_\alpha
 \brA_{L^2}^2=\blA \D^{1,\phi}\nabla_x F_\alpha \brA_{L^2}^2.
$$
On the other hand, there exists a positive constant $C$ such that, 
for any function $u\in {H}^{1,\phi}(\xR^2)$, there holds
\begin{multline}\label{equivalence:D1}
\frac{1}{C}\underset{\xR^2\times \xR^2}{\iint}\frac{\la \delta_h^2 u(x)\ra^2}{\la h\ra^2}\kappa^2\left(\frac{1}{|h|}\right)\frac{\dh}{|h|^2}\dx\\ 
\le 
\blA \D^{1,\phi} u \brA_{L^2}^2\le 
C\underset{\xR^2\times \xR^2}{\iint}\frac{\la \delta_h^2 u(x)\ra^2}{\la h\ra^2}\kappa^2\left(\frac{1}{|h|}\right)\frac{\dh}{|h|^2}\dx.
\end{multline}
To clarify notations, we recall that $\kappa^2(y)=(\kappa(y))^2$ and  $\delta_h^2 u=\delta_h(\delta_h u )$ where recall that $\delta_h g(x)=g(x)-g(x-h)$. 
It follows from the previous observations that
$$
Q\les \iiint \frac{\bla \delta_h^2 \nabla_x F_\alpha(x)\bra^2}{|h|^2}\kappa^2\left(\frac{1}{|h|}\right)\frac{\dx\dh}{|h|^2}\frac{\dalpha}{|\alpha|^{\frac{1}{2}}}\cdot
$$

We now have to bound $\bla \delta_h^2 \nabla_x F_\alpha(x)\bra^2$. To do so, we use the following analogue of the Leibniz rule:
$$
\delta_h(uv)=u(\delta_h v)+(\delta_h u)(\tau_h v) \quad\text{where}\quad 
\tau_hv(x)=v(x-h).
$$
With this property, 
it is easy to check that $\la \delta_h^2\nabla_x F_\a\ra\les q_1+\cdots+q_4$ where
\begin{align*}
%&\quad &&\text{where} \\
&q_1=\la \delta_h^2\nabla_x\Delta_\a f\ra,\quad &&q_2=\la\delta_h^2\Delta_\a f\ra\la\tau_h^2\nabla_x\Delta_\a f\ra,\\
&q_3=
\la \delta_h\Delta_\a f\ra\la\tau_h\delta_h\nabla_x\Delta_\a f\ra,\quad &&q_4=\la \delta_h\tau_h \Delta_\a f\ra\la\delta_h\Delta_\a f\ra\la \tau_h^2\Delta_\a\nabla_x f\ra.
\end{align*}
It follows that $Q\les Q_1+\cdots +Q_4$ with
$$
Q_j=\iiint \frac{q_j^2(x,h,\a)}{|h|^2}\kappa^2\left(\frac{1}{|h|}\right)\frac{\dx\dh}{|h|^2}\frac{\dalpha}{|\alpha|^{\frac{1}{2}}}\qquad (1\le j\le 4).
$$

Using the equivalence~\e{equivalence:D1} and Plancherel's identity, 
we find that
\be\label{est:Q1}
Q_1\les \int\blA \D^{1,\phi}\nabla_x\Delta_\a f\brA_{L^2}^2\frac{\dalpha}{|\alpha|^{\frac{1}{2}}}=\int\blA \D^{2,\phi}\delta_\a f\brA_{L^2}^2\frac{\dalpha}{|\alpha|^{2+\frac{1}{2}}}
\sim \blA \D^{2,\phi}f\brA_{\dot{H}^{\frac{1}{4}}}^2.
\ee

$\bullet$ To estimate the term $Q_2$, we write
\begin{align*}
Q_2&=\iiint \frac{\la\delta_h^2\Delta_\a f\ra^2\la\tau_h^2\nabla_x\Delta_\a f\ra^2(x,h,\a)}{|h|^2}\kappa^2\left(\frac{1}{|h|}\right)\frac{\dx\dh}{|h|^2}\frac{\dalpha}{|\alpha|^{\frac{1}{2}}}\\
&\le \iint \lA\tau_h^2\nabla_x\Delta_\a f\rA_{L^4_x}^2\frac{\lA\delta_h^2\Delta_\a f\rA_{L^4_x}^2}
{|h|^2}\kappa^2\left(\frac{1}{|h|}\right)\frac{\dh}{|h|^2}\frac{\dalpha}{|\alpha|^{\frac{1}{2}}}\\
&\le \int \lA\nabla_x\Delta_\a f\rA_{\dot{H}^\mez_x}^2\bigg( \int\frac{\lA\delta_h^2\Delta_\a f\rA_{\dot{H}^\mez_x}^2}
{|h|^2}\kappa^2\left(\frac{1}{|h|}\right)\frac{\dh}{|h|^2}\bigg)\frac{\dalpha}{|\alpha|^{\frac{1}{2}}}\cdot
\end{align*}
Then we estimate the innermost integral by applying~\e{equivalence:D1} 
with $u=\D^\mez \Delta_\a f$. This implies that
\begin{align*}
Q_2&\les  \int  \lA\nabla_x\Delta_\a f\rA_{\dot{H}^\mez_x}^2 \blA \Delta_\a \D^{\tdm,\phi}f\brA_{L^2}^2\frac{\dalpha}{|\alpha|^{\frac{1}{2}}}
=\int  \frac{\lA\nabla_x\delta_\a f\rA_{\dot{H}^\mez_x}^2}{|\a|^2}\frac{\blA \delta_\a \D^{\tdm,\phi}f\brA_{L^2}^2}{|\a|^2}\frac{\dalpha}{|\alpha|^{\frac{1}{2}}}\\
&\le \bigg(\int  \lA\nabla_x\delta_\a f\rA_{\dot{H}^\mez_x}^4\frac{\dalpha}{|\alpha|^{2+4\frac{5}{8}}}\bigg)^\mez
\bigg(\int\blA \delta_\a \D^{\tdm,\phi}f\brA_{L^2}^4\frac{\dalpha}{|\alpha|^{2+4\frac{5}{8}}}\bigg)^\mez.
\end{align*}
Now, the elementary estimate~\e{tech:1} implies that
\begin{align*}
&\bigg(\int  \lA\nabla_x\delta_\a f\rA_{\dot{H}^\mez_x}^4\frac{\dalpha}{|\alpha|^{2+4\frac{5}{8}}}\bigg)^\mez
\les\lA \nabla_x f\rA_{\dot{H}^{\mez+\frac{5}{8}}}^2\le  \lA f\rA_{\dot{H}^{\frac{17}{8}}}^2,\\
&\bigg(\int\blA \delta_\a \D^{\tdm,\phi}f\brA_{L^2}^4\frac{\dalpha}{|\alpha|^{2+4\frac{5}{8}}}\bigg)^\mez\les 
\blA \D^{\tdm,\phi}f\brA_{\dot{H}^{\frac{5}{8}}}^2=\blA \D^{\frac{17}{18},\phi}f\brA_{L^2}^2.
\end{align*}
This proves that
$$
Q_2\les \lA f\rA_{\dot{H}^{\frac{17}{8}}}^2\blA \D^{\frac{17}{18},\phi}f\brA_{L^2}^2.
$$

$\bullet$ To estimate the term $Q_3$, we apply the Cauchy-Schwarz inequality (and we balance the powers of $h$ and $\a$ in an appropriate way), to obtain $Q_3\le I_3 J_3$ where
\begin{align*}
I_3&=\bigg(\iiint \frac{\la\tau_h\delta_h\nabla_x
\delta_\a f\ra^4}{|h|^2 \la\a\ra^2}\frac{\dx\dh}{|h|^2}\frac{\dalpha}{|\alpha|^{\frac{1}{2}}}\bigg)^\mez,\\
J_3&=\bigg(\iiint \frac{\la \delta_h\delta_\a f\ra^4}{|h|^2 \la\a\ra^6}
\kappa^4\left(\frac{1}{|h|}\right)\frac{\dx\dh}{|h|^{2}}\frac{\dalpha}{|\alpha|^{\frac{1}{2}}}\bigg)^\mez.
\end{align*}
Since $\dot{H}^\mez(\xR^2)\subset L^4(\xR^2)$, we have
$$
J_3\les \bigg(\iiint \frac{\lA \delta_h\delta_\a f\rA_{\dot{H}^{\mez}}^4}{|h|^2 \la\a\ra^6}\kappa^4\left(\frac{1}{|h|}\right)\frac{\dh}{|h|^{2}}\frac{\dalpha}{|\alpha|^{\frac{1}{2}}}\bigg)^\mez.
$$
Now recall that, for all $a\in [0,+\infty)$, all $\gamma\ge 1$ and all $b,c\in (0,\gamma)$, there exists $C>0$ such that
\be\label{tech:5bis}
\iint_{\xR^2\times\xR^2} \Vert \delta_\alpha \delta_{h}f\Vert _{\dot H^{a}}^{2\gamma}
\kappa^{2\gamma}\left(\frac{1}{|h|}\right)\frac{\dh}{|h|^{2+2b}}\frac{\dalpha}{|\alpha|^{2+2c}}
\le C \blA \D^{a+\frac{b+c}{\gamma},\phi}f\brA_{L^2}^2.
\ee
Consequently, by applying this estimate with 
$(a,b,c,\gamma)=(1/2,1,7/4,2)$,
$$
J_3\les \blA \D^{\frac{17}{8},\phi}f\brA_{L^2}^2.
$$
Similarly, by using~\e{tech:5bis} with $(a,b,c,\gamma)=(3/2,1,1/4,2)$ and $\kappa\equiv 1$, we verify that
$$
I_3\les \lA f\rA_{\dot{H}^{\frac{17}{8}}}^2.
$$
This proves that 
$$
Q_3\les \lA f\rA_{\dot{H}^{\frac{17}{8}}}^2\blA \D^{\frac{17}{18},\phi}f\brA_{L^2}^2.
$$

$\bullet$ It remains to estimate the term $Q_4$. 
Using H\"older's inequality ($L^8\cdot L^8\cdot L^4 \hookrightarrow L^2$) as well as the Sobolev's inequalities 
($\dot{H}^\mez(\xR^2)\hookrightarrow L^4(\xR^2)$ and $\dot{H}^{\frac34}(\xR^2)\hookrightarrow L^8(\xR^2)$, we obtain 
\begin{align*}
Q_4
&=\iiint
\frac{\la \delta_h\tau_h \Delta_\a f\ra^2\la\delta_h\Delta_\a f\ra^2\la \tau_h^2\Delta_\a\nabla_x f\ra^2}{|h|^2}\kappa^2\left(\frac{1}{|h|}\right)\frac{\dx\dh}{|h|^2}\frac{\dalpha}{|\alpha|^{\frac{1}{2}}}\\
&\le \iint \lA \delta_h\tau_h \Delta_\a f\rA_{L^8}^2
\lA\delta_h\Delta_\a f\rA_{L^8}^2\lA \tau_h^2\Delta_\a\nabla_x f\rA_{L^4}^2
\kappa^2\left(\frac{1}{|h|}\right)\frac{\dh}{|h|^{4}}\frac{\dalpha}{|\alpha|^{\frac{1}{2}}}\\
&\les \int \lA \Delta_\a\nabla_x f\rA_{\dot{H}^\mez}^2\bigg(\int \lA \delta_h \Delta_\a f\rA_{\dot{H}^\frac34}^4\kappa^2\left(\frac{1}{|h|}\right)\frac{\dh}{|h|^4}\bigg)\frac{\dalpha}{|\alpha|^{\frac{1}{2}}}.
\end{align*}
Now, the innermost integral is estimated by
\begin{align*}
&\int \lA \delta_h \Delta_\a f\rA_{\dot{H}^\frac34}^4\kappa^2\left(\frac{1}{|h|}\right)\frac{\dh}{|h|^4}\\
&\qquad\qquad\le \bigg(\int \lA \delta_h \Delta_\a f\rA_{\dot{H}^\frac34}^4 \frac{\dh}{|h|^{2+4\frac12}}\bigg)^\mez
\bigg(\int \lA \delta_h \Delta_\a f\rA_{\dot{H}^\frac34}^4 \kappa^4\left(\frac{1}{|h|}\right)\frac{\dh}{|h|^{2+4\frac12}}
\bigg)^\mez\\
&\qquad\qquad\les \lA \Delta_\a f\rA_{\dot{H}^\frac54}^2\blA \D^{\frac54,\phi}\Delta_\a f\brA_{L^2}^2.
\end{align*}
Consequently,
\begin{align*}
Q_4&\les \int \lA \Delta_\a  \nabla_xf\rA_{\dot{H}^\mez}^2\lA \Delta_\a  f\rA_{\dot{H}^\frac54}^2
\blA \D^{\frac54,\phi}\Delta_\a f\brA_{L^2}^2\frac{\dalpha}{|\alpha|^{\frac{1}{2}}}\\
&\le\int \lA \delta_\a  f\rA_{\dot{H}^\tdm}^2\lA \delta_\a  f\rA_{\dot{H}^\frac54}^2
\blA \D^{\frac54,\phi}\delta_\a f\brA_{L^2}^2\frac{\dalpha}{|\alpha|^{6+\frac{1}{2}}}.
\end{align*}
It follows that $Q_4$ is bounded by
$$
\bigg[\int \lA \delta_\a  f\rA_{\dot{H}^\tdm}^8\frac{\dalpha}{|\a|^{2+8\frac{7}{12}}}\bigg]^\frac14
 \bigg[\int \lA \delta_\a  f\rA_{\dot{H}^\frac54}^8\frac{\dalpha}{|\a|^{2+8\frac{10}{12}}}\bigg]^\frac14
\bigg[\int \blA \D^{\frac54,\phi}\delta_\a f\brA_{L^2}^4\frac{\dalpha}{|\alpha|^{2+4\frac{10}{12}}}\bigg]^\mez.
$$
It follows from Lemma~\ref{L:technical}
that,
$$
Q_4\les \blA f\brA_{\dot{H}^{\frac{25}{12}}}^4\blA \D^{\frac{25}{12},\phi}f\brA_{L^2}^2.
$$
The estimate \eqref{Z11} is proven.

\textit{Step 2:} We now move to the proof of the estimate~\e{Z13}. 
We want to estimate the term 
$$
\Gamma_\alpha\defn\D^{\frac{3}{2},\phi}\left[F_\alpha G_\alpha \right]-F_\alpha\D^{\frac{3}{2},\phi}
G_\a - G_\a\D^{\frac{3}{2},\phi}\left[F_\alpha\right].
$$
The key point will be to exploit a representation of the operator $\D^{\frac{3}{2},\phi}$ in terms of finite differences. 
To do so, given $h\in \xR^2$, 
let us introduce the finite difference operator $s_h$ defined by
$$
(s_h f)(x)=2f(x)-f(x-h)-f(x+h).
$$
With this notation, the identity \e{dsphi} reads
\begin{equation*}
\D^{\frac{3}{2},\phi}g(x)=\int\frac{s_hg(x)}{|h|^{3/2}}\kappa\left(\frac{1}{|h|}\right)\frac{\dh}{|h|^2}\cdot
\end{equation*}
Now, it is easily verified that for any functions $u$ and $v$,
$$
s_h(uv)-u(s_h v)-v(s_hu)=(\delta_h u)(\delta_{h}v)-(\delta_{-h} u)(\delta_{-h}v).
$$
Hence, we obtain that
\begin{align*}
\la\Gamma_\alpha(x)\ra&\les \int \la \delta_h F_\a(x)\ra\la \delta_{h}G_\a(x)\ra
\kappa\left(\frac{1}{|h|}\right)\frac{\dh}{|h|^{2+\tdm}}\\
&\lesssim |\alpha|\int  \left|\Delta_\alpha \delta_{h}f(x)\right|
\left|\nabla_x\Delta_\alpha \delta_{h}f(x)\right|
\kappa\left(\frac{1}{|h|}\right)\frac{\dh}{|h|^{2+\tdm}}\cdot
\end{align*}
So, using successively Minkowski's inequality, H\"older's inequality and 
Sobolev's inequality, we find that
\begin{align*}
(II)&\lesssim\left( \iint \Vert \Delta_\alpha \delta_{h}f\Vert _{L^8}
\Vert \nabla_x\Delta_\alpha \delta_{h}f\Vert _{L^{\frac{8}{3}}}
\kappa\left(\frac{1}{|h|}\right)\frac{\dh}{|h|^{2+\tdm}}\frac{\dalpha}{|\alpha|}\right)^2
\\
&\lesssim\left( \iint \Vert \Delta_\alpha \delta_{h}f
\Vert _{\dot H^{\frac{3}{4}}}\Vert \Delta_\alpha \delta_{h}f\Vert _{\dot H^{\frac{5}{4}}}
\kappa\left(\frac{1}{|h|}\right)\frac{\dh}{|h|^{2+\tdm}}
\frac{\dalpha}{|\alpha|}\right)^2.
\end{align*}	
By using the Cauchy-Schwarz inequality, it follows that $(II)$ is bounded by 
\begin{align*}
C\bigg[\iint \Vert \delta_\alpha \delta_{h}f\Vert _{\dot H^{\frac{3}{4}}}^2
\kappa^2\left(\frac{1}{|h|}\right)\frac{\dh}{|h|^{2+\frac{19}{10}}}\frac{\dalpha}{|\alpha|^{2+\frac{11}{10}}}\bigg]\bigg[
\iint \Vert \delta_\alpha \delta_{h}f\Vert _{\dot H^{\frac{5}{4}}}^2
\frac{\dh}{|h|^{2+\frac{11}{10}}}\frac{\dalpha}{|\alpha|^{2+\frac{9}{10}}}\bigg].
\end{align*}
The estimate \e{tech:5bis} applied with $(a,b,c,\gamma)=(3/4,19/20,11/20,1)$ implies that
$$
\iint \Vert \delta_\alpha \delta_{h}f\Vert _{\dot H^{\frac{3}{4}}}^2
\kappa^2\left(\frac{1}{|h|}\right)\frac{\dh}{|h|^{2+\frac{19}{10}}}\frac{\dalpha}{|\alpha|^{2+\frac{11}{10}}}
\les \blA \D^{\frac{9}{4},\phi}f\brA_{L^2}.
$$
On the other hand, by applying twice the estimate~\e{tech:1}, we get
\begin{align*}
\iint \Vert \delta_\alpha \delta_{h}f\Vert _{\dot H^{\frac{5}{4}}}^2\frac{\dh}{|h|^{2+\frac{11}{10}}}\frac{\dalpha}{|\alpha|^{2+\frac{9}{10}}}
\les \lA f\rA_{\dot{H}^{\frac{9}{4}}}^2.
\end{align*}
This completes the proof of the claim~\e{Z13}, which in turn completes the proof of Proposition~\ref{Z18}.
\end{proof}

\subsection{End of the proof of Theorem \ref{mainthm0}}
We are now in position to complete the proof of Theorem~\ref{mainthm0}. 
Set
$$
A=\blA \D^{2,\phi}f\brA_{L^2}^2\quad , \quad B=\blA \D^{\frac{5}{2},\phi}f\brA_{L^2}^2.
$$
Write
\begin{align*}
&\int_{\xR^2}\Lr(f)f \D^{4,\phi^2} f\dx\\
&\qquad\qquad=\int_{\xR^2}\Big( \Lr(f)\big( \D^{\tdm,\phi}f\big) \D^{\frac52,\phi} f +\big[\D^{\tdm,\phi}, \Lr(f)\big]f  \D^{\frac52,\phi} f\Big)\dx\\
&\qquad\qquad\geq  \big\langle P(f)(\D^{\frac{3}{2},\phi}f),\D^{\frac{5}{2},\phi}f\big\rangle\\
&\qquad\qquad\quad-\left(\blA\left[\mathcal{R}, V(f)\right ] \D^{\frac{5}{2},\phi}f\brA_{L^2}+\lA H\rA_{L^2}\right) B^{\frac{1}{2}},
\end{align*}
where, by Lemma \ref{lem1}, 
\begin{align*}
	\big\langle P(f)(\D^{\frac{3}{2},\phi}f),\D^{\frac{5}{2},\phi}f\big\rangle
	&\geq \frac{B}{\langle \lA f\rA_{\dot W^{1,\infty}}\rangle^3}
	-C\lA f\rA_{\dot H^{\frac{9}{4}}}\blA \D^{\frac{19}{8},\phi}f\brA_{L^2}^2\\
	&\geq \frac{B}{\langle \lA f\rA_{\dot W^{1,\infty}}\rangle^3} 	-C\lA f\rA_{\dot H^{\frac{9}{4}}} A^{\frac{1}{4}}B^{\frac{3}{4}},
\end{align*}
and, by Proposition \ref{X1},
 \begin{align*}
 \blA \left[\mathcal{R}, V(f)\right ] (\D^{\frac{5}{2},\phi}f)\brA_{L^2} 
 &\lesssim  \left(\lA f\rA_{\dot H^{\frac{9}{4}}}+
 \lA f\rA_{\dot H^{\frac{17}{8}}}^2\right)\blA \D^{\frac{9}{4},\phi}f\brA_{L^2}	\\&\lesssim \left(\lA f\rA_{\dot H^{\frac{9}{4}}}+
 \lA f\rA_{\dot H^{\frac{17}{8}}}^2\right) A^{\frac{1}{4}}B^{\frac{1}{4}},
 \end{align*}
and, by Proposition \ref{Z19} and \ref{Z18}
\begin{align*}
\lA H\rA&\leq \blA R(f,\D^{\frac{3}{2},\phi}f)\brA_{L^2}+\blA \big[ \D^{\frac{3}{2},\phi},\mathcal{L}(f)\big](f)\brA_{L^2}\\&\lesssim\lA f\rA_{\dot H^{\frac{9}{4}}}\left(\Vert \D^{\frac{9}{4},\phi}f\Vert _{L^2}+\lA f\rA_{\dot{H}^{\frac{17}{8}}}\blA \D^{\frac{17}{18},\phi}f\brA_{L^2}+
\lA f\rA_{\dot{H}^{\frac{25}{12}}}^2\blA \D^{\frac{25}{12},\phi}f\brA_{L^2}\right)
\\&\lesssim\lA f\rA_{\dot H^{\frac{9}{4}}}\left(A^{\frac{1}{4}}B^{\frac{1}{4}}+\lA f\rA_{\dot{H}^{\frac{17}{8}}} A^{\frac{3}{8}}B^{\frac{1}{8}}+
\lA f\rA_{\dot{H}^{\frac{25}{12}}}^2A^{\frac{5}{12}}B^{\frac{1}{12}}\right).	
\end{align*}
By  Lemma  \ref{L:3.5}, one has for $s\in (2,5/2)$,
$$
	\lA f\rA_{\dot H^s}\lesssim \left(\phi\left(\frac{B}{A}\right)\right)^{-1} A^{\frac52-s}  B^{s-2}.
$$
Thus, 
\begin{align*}
\lA f\rA_{\dot H^{\frac{9}{4}}} A^{\frac{1}{4}}B^{\frac{3}{4}}+	\left(\lA\left[\mathcal{R}, V(f)\right ] \D^{\frac{5}{2},\phi}f\rA_{L^2}+\lA H\rA_{L^2}\right)B^{\frac{1}{2}}\lesssim A^{\frac{1}{2}}(1+A) B\left[\phi\left(\frac{B}{A}\right)\right]^{-1}.
\end{align*}
This implies \eqref{Z5} and hence complete the proof.

\section{Well-posedness results}

In this section we complete the proofs of Theorem~\ref{main}, 
Theorem~\ref{main-2} and Theorem~\ref{theo:main3}. 
To do so, we proceed by 
following a classical strategy:
\begin{enumerate}
\item We define approximate 
systems for which the existence of smooth solutions 
is easily obtained. 
\item We then prove uniform estimates for the solutions of the approximate systems, on a uniform time interval. 
\item Finally, we verify that 
the sequence of approximate solutions 
converges to a solution of the Muskat equation 
and prove a 
uniqueness result. 
\end{enumerate}
As we will see, this section does not contain any new arguments. 
Indeed, the core of the whole argument is 
contained in the 
estimates already proven in the previous sections. 
We will also need additional estimates, but they will all be obtained by adapting 
arguments similar to ones in our previous papers~\cite{AN1,AN2,AN3} for the 2D Muskat equation. 
For the sake of readability, we will state the corresponding estimates for the 3D problem and 
recall the principles of their proofs.

\subsection{Functional spaces}
We begin by recalling 
the definition of the 
spaces in which we shall work to study the Cauchy problem 
(see \S\ref{S:2.3}).

\begin{definition}\label{defi:D}
Consider 
a function $\phi\colon [0,\infty)\to [1,\infty)$, of the 
form
\begin{equation}\label{n10-ter}
\phi(\lambda)=4\pi\int_0^\infty \frac{1-\cos(r)}{r^{3/2}}\kappa\left(\frac{\lam}{r}\right)\frac{\dr}{r},
\end{equation}
where $\kappa\colon[0,\infty) \to [1,\infty)$ is an admissible weight (see Definition~\ref{defi:admissible}). 
\end{definition}

\subsection{Approximate Cauchy problems}
 To prove existence, we introduce the following 
 Cauchy problem depending on the parameter $\eps\in (0,1]$: 
\be\label{n2}
\left\{
\begin{aligned}
&\partial_tf-|\log(\varepsilon)|^{-1}\Delta f
= 	N_\varepsilon(f),\\
&f\arrowvert_{t=0}=f_0\star \chi_\eps,
\end{aligned}
\right.
\ee
where
\begin{align*}
	N_\varepsilon(f)= \frac{1}{2\pi}\int_{\mathbb{R}^2}
	\frac{\alpha\cdot\nabla_x\Delta_\alpha f(x)}{\langle \Delta_\alpha f(x)\rangle^{3}}
	\left(1-\chi\left(\frac{|\alpha|}{\varepsilon}\right)\right)\frac{\dalpha}{|\alpha|^2}
\end{align*}
and  $\chi_\eps(x)=\eps^{-1}\chi(x/\eps)$ where 
$\chi$ is a smooth bump function satisfying $0\le \chi\le 1$ and
$$
\chi(y)=\chi(-y),\quad 
\chi(y)=1 \quad\text{for}\quad |y|\le \frac14, \quad \chi(y)=0 \quad\text{for }\la y\ra\ge 2,\quad 
\int_{\xR^2}\chi \dy=1.
$$

\begin{lemma}\label{P:initiale}
For any $\eps$ in $(0,1]$ and any initial data $f_0$ 
in $H^{2}(\xR^2)$, there exists a unique global in time solution $f_\eps$ satisfying
$$
f_\eps\in C^1([0,+\infty);H^\infty(\xR^2)).
$$
\end{lemma}
\begin{proof}
See Section~\S$2.9$ in \cite{AN2}.
\end{proof}

\begin{lemma}\label{L:2.1}
For any $\beta_0\in (0,1/2)$, there exists 
$c_0>0$ such that, for any $\eps\in (0,1]$ 
and any 
solution $f\in C^1([0,T];H^{\infty}(\xR^2))$ of the approximate Muskat equation \eqref{n2},
\be\label{a1-AN4}
\fract \lA f(t)\rA_{L^{\infty}}\le c_0 \varepsilon^{\beta_0}\lA f(t)\rA_{\dot C^{1+\beta_0}},
\ee
\be\label{a2-AN4}
\fract \lA \nabla f(t)\rA_{L^{\infty}}\le c_0\lA f(t)\rA_{\dot H^{\frac52}}^2
+c_0\varepsilon^{\beta_0}\lA f(t)\rA_{\dot C^{2+\beta_0}}.
\ee
\end{lemma}
\begin{proof}1) It follows from the proof of Theorem $4.1$ in~\cite{CG-CMP2} that
\begin{align*}
\fract \lA f(t)\rA_{\dot L^{\infty}}\le \sup_{x\in \mathbb{R}^2}
\frac{1}{2\pi}\left|\int_{\mathbb{R}^2}\frac{\alpha\cdot\nabla_x\Delta_\alpha f(t,x)}{\langle \Delta_\alpha f(t,x)\rangle^{3}}
\chi\left(\frac{|\alpha|}{\varepsilon}\right)\frac{\dalpha}{|\alpha|^2}\right|
\lesssim \varepsilon^{\beta_0}\lA f(t)\rA_{\dot C^{1+\beta_0}},
\end{align*}
which gives \eqref{a1-AN4}.

2) On the other hand, it follows from \cite[ Section 10]{Gancedo-Lazar-H2} that 
\begin{align*}
\fract \lA \nabla f(t)\rA_{L^{\infty}}&\le \frac{1}{2\pi} \lA f(t)\rA_{\dot H^{\frac52}}^2
+
\sup_{x\in \mathbb{R}^2} \frac{1}{2\pi}
\left|\nabla_x\int_{\mathbb{R}^2}
\frac{\alpha\cdot\nabla_x\Delta_\alpha f(t,x)}{\langle \Delta_\alpha f(t,x)\rangle^{3}}
\chi\left(\frac{|\alpha|}{\varepsilon}\right)\frac{\dalpha}{|\alpha|^2}\right|\\
&\lesssim  \lA f(t)\rA_{\dot H^{\frac52}}^2+ \varepsilon^{\beta_0}\lA f(t)\rA_{\dot C^{2+\beta_0}},
\end{align*}
which is the wanted estimate~\eqref{a2-AN4}. 
\end{proof}
In view of \e{a2-AN4}, we also need 
to estimate the H\"older norm $\lA f(t)\rA_{\dot C^{2+\beta_0}}$. 
This is the purpose of the following result.
\begin{lemma} There exists a positive constant $K$, independent of $\eps$, such that, for any $0<\beta_1\leq 10^{-4}$,
	\begin{align}\label{Z1}
\lA N_\varepsilon(f)\rA_{\dot C^{\beta_1}}&\le K \lA f\rA_{ C^{1+2\beta_1}}\left(1+\lA f\rA_{ C^{1+2\beta_1}}\right),
	\end{align}
where $\lA f\rA_{ C^{1+2\beta_1}}= \lA f\rA_{L^\infty}+	\lA f\rA_{\dot  C^{1+2\beta_1}}$. 
In particular, 
\begin{align}\label{Z2}
	\int_{0}^{t}\lA f(\tau)\rA_{\dot C^{2+\frac{\beta_1}{2}}} \dtau&\lesssim 	|\log(\varepsilon)|^{2}
	t^{\frac{2-\beta_1}{4}}\lA \nabla f_0\rA_{L^\infty}\\&~~+ |\log(\varepsilon)|^{2}t^{\frac{\beta_1}{4}} \int_{0}^{t} \lA f(\tau)\rA_{ C^{1+2\beta_1}}\left(1+\lA f(\tau)\rA_{ C^{1+2\beta_1}}\right)\dtau.\nonumber
\end{align}
\end{lemma}
\begin{proof} One has for $\alpha'\in \mathbb{R}^2,$
	\begin{align*}
		|\delta_{\alpha'} N_\varepsilon (f)(x)|&\lesssim  \left|\int_{\mathbb{R}^2}
		\frac{\alpha\cdot\nabla_x\Delta_\alpha \delta_{\alpha'}f(x)}{\langle \Delta_\alpha f(x-\alpha')\rangle^{3}}
		\left(1-\chi\left(\frac{|\alpha|}{\varepsilon}\right)\right)\frac{\dalpha}{|\alpha|^2}\right|\\&\quad+\int_{\mathbb{R}^2} \la\alpha\cdot\nabla_x\Delta_\alpha f(x)\ra
	\la	\delta_{\alpha'}\left(\frac{1}{\langle \Delta_\alpha f(\cdot)\rangle^{3}}\right)(x)\ra\frac{\dalpha}{|\alpha|^2}.	\end{align*}
Since
\begin{align*}
\left|\frac{1}{\langle a\rangle^3}-1\right|\lesssim |a|, 	\quad
\left|\frac{1}{\langle a\rangle^3}-\frac{1}{\langle b\rangle^3}\right|\lesssim|a-b|,
\end{align*}
we obtain, for $0<\beta_2\leq \frac{\beta_1}{100}$ and for any  $\alpha'\in \mathbb{R}^2$, 
\begin{align*}
|\delta_{\alpha'} N_\varepsilon (f)(x)|
&\lesssim  \left|\int_{\mathbb{R}^2}
\alpha\cdot\nabla_x\Delta_\alpha \delta_{\alpha'}f(x)
\left(1-\chi\left(\frac{|\alpha|}{\varepsilon}\right)\right)\frac{\dalpha}{|\alpha|^2}\right|\\
&\quad+\int
|\nabla_x\delta_\alpha \delta_{\alpha'}f(x)||\delta_{\alpha} f(x)|\frac{\dalpha}{|\alpha|^3}+\int |\nabla_x\delta_\alpha f(x)||\delta_\alpha \delta_{\alpha'}f(x)|\frac{\dalpha}{|\alpha|^3}\\
&\lesssim 
\lA \delta_{\alpha'}\D f\rA_{L^\infty}+\varepsilon^{\beta_2}\sup_{\alpha}\frac{|\delta_{\alpha} (\nabla_x \delta_{\alpha'} f)(x)|}{|\alpha|^{\beta_2}}\\
&\quad+\sup_{\alpha}\frac{|\delta_{\alpha} (\nabla \delta_{\alpha'} f)(x)|}{|\alpha|^{\beta_2}} \int \frac{|\delta_{\alpha}f(x)|}{|\alpha|^{3-\beta_2}}\dalpha\\
&\quad+\sup_{\alpha}\frac{|\delta_\alpha \delta_{\alpha'}f(x)|}{|\alpha|^{1-\beta_2}}\int \frac{|\nabla_x\delta_\alpha f(x)|}{|\alpha|^{2+\beta_2}}\dalpha.
\end{align*}
This implies, 
\begin{align*}
\lA N_\varepsilon(f)\rA_{\dot C^{\beta_1}}
&\lesssim \lA f\rA_{\dot C^{1+\beta_1-\beta_2}}^{\frac{1}{2}}\lA f\rA_{\dot C^{1+\beta_1+\beta_2}}^{\frac{1}{2}}\\
&\quad+\varepsilon^{\beta_2}\sup_{\alpha,\alpha',x}\frac{|\delta_{\alpha} (\nabla \delta_{\alpha'} f)(x)|}{|\alpha|^{\beta_2}|\alpha'|^{\beta_1}}\\
&\quad+\sup_{\alpha,\alpha',x}\frac{|\delta_{\alpha} (\nabla \delta_{\alpha'} f)(x)|}{|\alpha|^{\beta_2}|\alpha'|^{\beta_1}}\lA f\rA_{L^\infty}^{\beta_2}\lA \nabla f\rA_{L^\infty}^{1-\beta_2}\\
&\quad+\sup_{\alpha,\alpha',x}\frac{|\delta_\alpha \delta_{\alpha'}f(x)|}{|\alpha|^{1-\beta_2}|\alpha'|^{\beta_1}}\lA \nabla f\rA_{L^\infty}^{\frac{1}{2}}\lA f\rA_{\dot C^{1+2\beta_2}}^{\frac{1}{2}}.
\end{align*}
It is easy to check that 
\begin{align*}
	|\delta_{\alpha} \delta_{\alpha'} g(x)|\lesssim |\alpha|^{a_1}|\alpha|^{a_2}\lA g\rA_{\dot C^{a_1+a_2-\eps_0}}^{\frac{1}{2}}\lA g\rA_{\dot C^{a_1+a_2+\eps_0}}^{\frac{1}{2}},
\end{align*}
for any $0<\eps_0<(a_1+a_2)/100$.
Therefore, 
\begin{align*}
\lA N_\varepsilon(f)\rA_{\dot C^{\beta_1}}&\lesssim \lA f\rA_{\dot C^{1+\beta_1-\beta_2}}^{\frac{1}{2}}\lA f\rA_{\dot C^{1+\beta_1+\beta_2}}^{\frac{1}{2}}+\varepsilon^{\beta_2}\lA f\rA_{\dot C^{1+\beta_1}}^{\frac{1}{2}}\lA f\rA_{\dot C^{1+\beta_1+2\beta_2}}^{\frac{1}{2}}\\&\quad+\lA f\rA_{\dot C^{1+\beta_1}}^{\frac{1}{2}}\lA f\rA_{\dot C^{1+\beta_1+2\beta_2}}^{\frac{1}{2}}\lA f\rA_{L^\infty}^{\beta_2}\lA \nabla f\rA_{L^\infty}^{1-\beta_2}\\
&\quad+\lA f\rA_{\dot C^{1+\beta_1-2\beta_2}}^{\frac{1}{2}}\lA f\rA_{\dot C^{1+\beta_1}}^{\frac{1}{2}}\lA \nabla f\rA_{L^\infty}^{\frac{1}{2}}\lA f\rA_{\dot C^{1+2\beta_2}}^{\frac{1}{2}}.
\end{align*}
By interpolation inequality, one gets \eqref{Z1}.
Now we observe that
\begin{align*}
f(t,x)&=c\int_{\mathbb{R}^2}\frac{1}{t}\exp\(-|\log \varepsilon|\frac{|x-y|^2}{4t}\) f_0(y)\dy\\&\quad\quad
+c\int_{0}^{t}\int_{\mathbb{R}^2}\frac{1}{t-\tau}\exp\(-|\log \varepsilon|\frac{|x-y|^2}{4(t-\tau)}\) N_\varepsilon(f)(\tau,y)\dy\dtau,
\end{align*}
so,
\begin{align*}
\lA f(t)\rA_{\dot C^{2+\frac{\beta_1}{2}}}
&\lesssim |\log(\varepsilon)|^{2}t^{-\frac{2+\beta_1}{4}}\lA \nabla f_0\rA_{L^\infty}\\
&\quad +|\log(\varepsilon)|^{2}
\int_{0}^{t}\frac{1}{(t-\tau)^{1-\beta_1/4}} \lA 
N_\varepsilon(f)(\tau)\rA_{\dot C^{\beta_1}}\dtau.
\end{align*}
This implies \eqref{Z2}. The proof is complete. 
\end{proof}
\begin{corollary}\label{Coro1} For $0<\beta_0<10^{-3}$ and 
$t\leq |\log(\varepsilon)|$, there holds
\begin{multline}\label{Lipes}
	\sup_{\tau\in [0,t]}\lA \nabla f(\tau)\rA_{L^{\infty}}-\lA \nabla f_0\rA_{L^{\infty}}\\
	\lesssim 
	\int_{0}^{t}\lA f(t)\rA_{\dot H^{\frac52}}^2
	+ \varepsilon^{\beta_0}\int_{0}^{t} \lA f(\tau)\rA_{\dot H^{2}}\dtau+\varepsilon^{\beta_0/4}(1+\lA f_0\rA_{W^{1,\infty}}).
\end{multline} 
\end{corollary}
\begin{proof} By \eqref{a2-AN4} and \eqref{Z2}, one has for any $t\leq |\log(\varepsilon)|$, 
	\begin{align*}
\sup_{\tau\in [0,t]}\lA \nabla f(\tau)\rA_{L^{\infty}}-\lA \nabla f_0\rA_{L^{\infty}}&\lesssim \int_{0}^{t}\lA f(t)\rA_{\dot H^{\frac52}}^2
+\varepsilon^{\beta_0}\int_{0}^{t}\lA f(t)\rA_{\dot C^{2,\beta_0}}\\&\lesssim \int_{0}^{t}\lA f(t)\rA_{\dot H^{\frac52}}^2
+\varepsilon^{\beta_0/2}\lA \nabla f_0\rA_{L^{\infty}}\\&\quad+ \varepsilon^{\beta_0/2}\int_{0}^{t} \lA f(\tau)\rA_{ C^{1+4\beta_0}}\left(1+\lA f(\tau)\rA_{ C^{1+4\beta_0}}\right)\dtau.
\end{align*}
By interpolation inequality,
\begin{align*}
\sup_{\tau\in [0,t]}\lA \nabla f(\tau)\rA_{L^{\infty}}-\lA \nabla f_0\rA_{L^{\infty}}&\lesssim  
\int_{0}^{t}\lA f(t)\rA_{\dot H^{\frac52}}^2
+\varepsilon^{\beta_0/2}\lA \nabla f_0\rA_{L^{\infty}}\\&\quad+ \varepsilon^{\beta_0/2}\int_{0}^{t} \lA f(\tau)\rA_{L^\infty}\dtau+\varepsilon^{\beta_0/2}.
\end{align*} 
Combining this with \eqref{a1-AN4}, one gets,
\begin{align*}
&\sup_{\tau\in [0,t]}\lA \nabla f(\tau)\rA_{L^{\infty}}-\lA \nabla f_0\rA_{L^{\infty}}\\
&\qquad\qquad\lesssim 
\int_{0}^{t}\lA f(t)\rA_{\dot H^{\frac52}}^2
+ \varepsilon^{\beta_0}\int_{0}^{t} \lA f(\tau)\rA_{\dot C^{1,\beta_0}}\dtau+\varepsilon^{\beta_0/4}(1+\lA f_0\rA_{W^{1,\infty}})\\
&\qquad\qquad\lesssim 
	\int_{0}^{t}\lA f(t)\rA_{\dot H^{\frac52}}^2
	+ \varepsilon^{\beta_0}\int_{0}^{t} \lA f(\tau)\rA_{\dot H^{2}}^{1-2\beta_0}\lA f(\tau)\rA_{\dot H^{\frac{5}{2}}}^{2\beta_0}\dtau
	+\varepsilon^{\beta_0/4}(1+\lA f_0\rA_{W^{1,\infty}}).
\end{align*} 
This gives \eqref{Lipes}. The proof is complete. 
\end{proof}
\begin{lemma} \label{Z6}For $\beta_0\in (0,10^{-10}]$
\begin{equation}\label{n68}
\lA \int_{\mathbb{R}^2}
	\frac{\alpha\cdot\nabla_x\Delta_\alpha f}{\langle \Delta_\alpha f\rangle^{3}}
	\chi\left(\frac{|\alpha|}{\varepsilon}\right)\frac{\dalpha}{|\alpha|^2}\rA_{\dot H^{\frac{9}{5}}}
	\lesssim \varepsilon^{2\beta_0} (1+\lA f\rA_{\dot H^2})^{3} (1+\lA f\rA_{\dot H^3}).
\end{equation}
\end{lemma}
\begin{proof}
Set 
\begin{align*}
B_\alpha(x):=	\left|\D^{\frac{9}{5}}\left(	\frac{\alpha\cdot\nabla_x\Delta_\alpha f(.)}{\langle \Delta_\alpha f(.)\rangle^{3}}\right) (x)\right|
\end{align*}
We have to prove that
\begin{align}\label{Z4}
\varepsilon^{2\beta_0}
\lA\int |B_\alpha(.)| \frac{\dalpha}{|\alpha|^{2+2\beta_0}}\rA_{L^2}\lesssim \varepsilon^{\beta_0} \lA\sup_{\alpha} \frac{|B_\alpha| }{|\alpha|^{\beta_0}}\rA_{L^2}^{\frac{1}{2}}\lA\sup_{\alpha} \frac{|B_\alpha| }{|\alpha|^{3\beta_0}}\rA_{L^2}^{\frac{1}{2}}
\end{align} We have
\begin{align*}
B_\alpha(x)&\lesssim |\delta_{\alpha} \nabla \D^{\frac{9}{5}} f(x)|+|\delta_{\alpha}\nabla f(x)|\left|\D^{\frac{9}{5}}\left(\frac{1}{\langle \Delta_\alpha f(.)\rangle^{3}}\right)\right|\\&+\int |\delta_{\alpha'}\delta_{\alpha}\nabla_x f(x)||\delta_{\alpha'}  \left(\frac{1}{\langle \Delta_\alpha f(.)\rangle^{3}}\right)(x)| \frac{\dalpha'}{|\alpha'|^{2+\frac{9}{5}}}.
	\end{align*} 
Using the inequalities
\begin{align*}
|\delta_{\alpha'}  \left(\frac{1}{\langle \Delta_\alpha f\rangle^{3}}\right)(x)|&\lesssim |\delta_{\alpha'} \Delta_\alpha f(x)|,\\
\left|\D^{\frac{9}{5}}\left(\frac{1}{\langle \Delta_\alpha f\rangle^{3}}\right)(x)\right|
&\lesssim |\Delta_{\alpha} \D^{\frac{9}{5}}f(x)|\\
&\quad+\int |\delta_{\alpha'}\Delta_{\alpha} f(x)|^2+|\delta_{\alpha'}\Delta_{\alpha} f(x)|^3\frac{\dalpha'}{|\alpha'|^{2+\frac{9}{5}}},
\end{align*}
one obtains,
\begin{align}\label{Z3}
B_\alpha(x)
&\lesssim |\delta_{\alpha} \nabla \D^{\frac{9}{5}} f(x)|+|\delta_{\alpha}\nabla f(x)||\Delta_{\alpha} \D^{\frac{9}{5}}f(x)|\\
&\nonumber\quad+|\delta_{\alpha}\nabla f(x)|\int |\delta_{\alpha'}\Delta_{\alpha} f(x)|^2+|\delta_{\alpha'}\Delta_{\alpha} f(x)|^3\frac{\dalpha'}{|\alpha'|^{2+\frac{9}{5}}}\\
&\quad+\sup_{\alpha'}\frac{|\delta_{\alpha'} \Delta_\alpha f(x)|}{|\alpha'|^{\frac{9}{10}}}\int |\delta_{\alpha'}\delta_{\alpha}\nabla_x f(x)| \frac{\dalpha'}{|\alpha'|^{2+\frac{9}{10}}}.\nonumber
\end{align} 
It is easy to check that, for any $0<\varepsilon_0\leq 10^{-4}$, $\alpha\in \mathbb{R}^2$, 
\begin{align*}
&	\int |\delta_{\alpha'}\delta_{\alpha}\nabla_x f(x)| \frac{\dalpha'}{|\alpha'|^{2+\frac{9}{10}}}\lesssim_{\varepsilon_0} \left(\sup_{\alpha'} \frac{|\delta_{\alpha'}\delta_{\alpha}\nabla_x f(x)|}{|\alpha'|^{\frac{9}{10}-\varepsilon_0}}\right)^{\frac{1}{2}}\left(\sup_{\alpha'} \frac{|\delta_{\alpha'}\delta_{\alpha}\nabla_x f(x)|}{|\alpha'|^{\frac{9}{10}+\varepsilon_0}}\right)^{\frac{1}{2}},\\&
\int |\delta_{\alpha'}\Delta_{\alpha} f(x)|^2\frac{\dalpha'}{|\alpha'|^{2+\frac{9}{5}}}\lesssim_{\varepsilon_0} \sup_{\alpha'}\frac{|\delta_{\alpha'}\Delta_{\alpha} f(x)|}{|\alpha'|^{\frac{9}{10}-\varepsilon_0}} \sup_{\alpha'}\frac{|\delta_{\alpha'}\Delta_{\alpha} f(x)|}{|\alpha'|^{\frac{9}{10}+\varepsilon_0}},\\
&\int |\delta_{\alpha'}\Delta_{\alpha} f(x)|^3\frac{\dalpha'}{|\alpha'|^{2+\frac{9}{5}}}\lesssim_{\varepsilon_0} \left(\sup_{\alpha'}\frac{|\delta_{\alpha'}\Delta_{\alpha} f(x)|}{|\alpha'|^{\frac{3}{5}-\varepsilon_0}}\right)^{\frac{3}{2}}\left(\sup_{\alpha'}\frac{|\delta_{\alpha'}\Delta_{\alpha} f(x)|}{|\alpha'|^{\frac{3}{5}+\varepsilon_0}}\right)^{\frac{3}{2}}\cdot
\end{align*} 
Combining these inequalities with \eqref{Z3} and \eqref{Z4}, we get \e{n68}. This completes the proof. 
\end{proof}
\subsection{A priori estimate}
In this paragraph, we gather the {\em a priori} estimates that can be deduced from the previous estimates for the nonlinearity. 
Introduce the time dependent functions
\begin{align*}
	A_\phi=\blA \D^{2,\phi}f\brA_{L^2}^2,\quad
	B_\phi=\blA \D^{\frac52,\phi}f\brA_{L^2}^2,\quad Z_\phi=\blA \D^{3,\phi}f\brA_{L^2}^2
\end{align*}

\begin{proposition}\label{mainthm}
Assume that the weight $\phi(r)\les \log(2+r)$. Then  
\begin{multline}\label{Z7}
\frac{1}{2}\fract A_\phi(t)
+\frac{B_\phi(t)}{1+ \lA \nabla f\rA_{L^\infty}^3}+|\log (\varepsilon)|^{-1}Z_\phi(t)\\
\lesssim  A_\phi^{\frac{1}{2}}(1+A_\phi)B_\phi \left(\phi\left(\frac{B_\phi}{A_\phi}\right)\right)^{-1}+
\varepsilon^{2\beta_0} (1+A_\phi)^{2} (1+Z_\phi(t))^{\frac{4}{5}}.
\end{multline}
\end{proposition}
\begin{proof}
Multiply \e{n2} by $\D^{4,\phi^2} f$ to obtain the identity
\begin{align*}
	&\frac{1}{2}	\fract \blA D^{2,\phi}f\brA_{L^2}^2
	+|\log (\varepsilon)|^{-1}\blA \D^{3,\phi} f\brA_{L^2}^2+	\int \Lr(f)f \D^{4,\phi^2} f \dx\\&=-
	\frac{1}{2\pi}\int \frac{\alpha\cdot\nabla_x\Delta_\alpha f(x)}{\langle \Delta_\alpha f(x)\rangle^{3}}
	\chi\left(\frac{|\alpha|}{\varepsilon}\right)\frac{\dalpha}{|\alpha|^2}  D|^{4,\phi^2} f \dx.
	\end{align*}
Thanks to Theorem \ref{mainthm0} and  Lemma \ref{Z6}, 
\begin{align*}
	\text{LHS of}~ \eqref{Z7}&\lesssim A_\phi^{\frac{1}{2}}(1+A_\phi)B_\phi \left(\phi\left(\frac{B_\phi}{A_\phi}\right)\right)^{-1}\\
	&\quad +
	\varepsilon^{2\beta_0} (1+A_\phi)^{\frac{3}{2}} (1+Z_\phi(t)^{\frac{1}{2}}) \lA\D^{\frac{11}{5},\phi^2} f\rA_{L^2}.
\end{align*}
Since $\phi(r)\les  \log(2+r)$ we have
\begin{align*}
\lA\D^{\frac{11}{5},\phi^2} f\rA_{L^2}
&\les \lA\D^{\frac{11}{5}} f\rA_{L^2}+\lA\D^{\frac{11}{5}+10^{-3}} f\rA_{L^2}\\
&\les (1+A_\phi)^\mez (1+Z_\phi)^{\frac{3}{10}}.
\end{align*}
This completes the proof.
\end{proof}

To conclude, it remains to control the factor 
$(1+\lA\nabla f\rA_{L^\infty})^{-1}$. 
To do so, we will use two different arguments depending on the assumptions on the initial data.\begin{proposition} 
\begin{enumerate}[(i)]
\item\label{item-i)} For any $t\leq |\log(\varepsilon)|$ with $\varepsilon\ll 1$ and for any 
$\beta_0=10^{-10}$, there holds
\begin{align}\label{Lipes2}
\sup_{\tau\in [0,t]}\lA \nabla f(\tau)\rA_{L^{\infty}}- \lA \nabla f_0\rA_{L^{\infty}}&\les
\int_{0}^{t}	B_\phi(\tau)\dtau\\
&\quad + \varepsilon^{\beta_0}\int_{0}^{t} A_\phi(\tau)\dtau+\varepsilon^{\beta_0/4}(1+\lA f_0\rA_{W^{1,\infty}}).\nonumber
\end{align} 
\item\label{item-ii)}If $\kappa(r)\ge \log(4+r)^a$ for some $a\ge 0$, then for $t\leq |\log(\varepsilon)|$
\begin{align*}
\lA \nabla f(t)\rA_{L^{\infty}}\lesssim 1+\lA f_0\rA_{L^\infty}+\varepsilon^{\beta_0} \int_{0}^{t} (A_\phi(\tau)+ Z_\phi(\tau))\dtau+A_\phi\log (2+B_\phi)^{\frac{1-2a}{2}}.
\end{align*}
\end{enumerate} 
\end{proposition}
\begin{proof} Statement $(\ref{item-i)})$ is obtained from Corollary \ref{Coro1}. On the other hand, by \eqref{linf'} and \eqref{a1-AN4} we have for any $t\leq |\log(\varepsilon)|$
	\begin{align*}
	\lA \nabla f(t)\rA_{L^{\infty}}
		&\lesssim 1+ \lA f(t)\rA_{L^\infty}+A_\phi\log (2+B_\phi)^{\frac{1-2a}{2}}\\&\lesssim  1+ \lA f_0\rA_{L^\infty}+\varepsilon^{\beta_0} \int_{0}^{t}\lA f(\tau)\rA_{\dot H^{2+\beta_0}}\dtau+C_0A_\phi\log (2+B_\phi)^{\frac{1-2a}{2}}
		\\&\lesssim 1+\lA f_0\rA_{L^\infty}+\varepsilon^{\beta_0} \int_{0}^{t} (A_\phi(\tau)+ Z_\phi(\tau))\dtau+A_\phi\log (2+B_\phi)^{\frac{1-2a}{2}}.
	\end{align*}
This implies $(\ref{item-ii)})$ which completes the proof.
\end{proof}
\subsection{Choice of the weight $\phi$}

We will apply the previous {\em a priori} estimate with three different choices 
for the weight function $\phi$. In particular, we chose 
\begin{enumerate}
\item $\phi=1$: to prove Theorem~\ref{main-2} about global well-posedness for small data in $\dot{H}^2(\xR^2)\cap W^{1,\infty}(\xR^2)$. 
\item $\phi(r)= \log(2+r)^a$: to prove Theorem~\ref{theo:main3} about local and global well-posedness in $H^{2,\log^a}(\xR^2)$.
\end{enumerate}
To prove Theorem~\ref{main}, about large data in $H^2(\xR^2)\cap W^{1,\infty}(\xR^2)$, following the strategy introduced in \cite{AN2}, 
we need to consider a weight 
$\phi=\phi_0$ depending on the initial data. To do so, we use the following

\begin{lemma}\label{L:critical}
For any function $f_0$ in $\dot{H}^2(\xR^2)$ there exists an admissible weight $\kappa$ 
such that $f_0$ belongs to the space $\dot{H}^{2,\phi}(\xR^2)$. 
\end{lemma}
\begin{proof}
It is proved\footnote{This result is proved in \cite{AN1} for $d=1$ only, 
but the proof clearly applies in any dimension.} in \cite{AN1} that, 
for any nonnegative integrable function 
$\omega\in L^1(\xR^2)$, there exists a function  $\eta\colon[0,\infty) \to [1,\infty)$ satisfying the following properties: 
\begin{enumerate}
\item $\eta$ is increasing and $\lim\limits_{r\to \infty}\eta(r)=\infty$,
\item $\eta(2r)\leq 2\eta(r)$ for any $r\geq 0$,
\item $\omega$ satisfies the enhanced integrability condition:
\begin{equation}\label{enhanced}
\int_{\xR^2} \eta(|x|) \omega(x) \dx<\infty,
\end{equation}
\item moreover, the function $r\mapsto \eta(r)/\log(4+r)$ is decreasing on  $[0,\infty)$.
\end{enumerate}
Then we apply this result with 
$\omega(\xi)=\la\xi\ra^4\bla\hat{f}(\xi)\bra^2$ and 
set $\kappa(r)=\sqrt{\eta(r)}$. Remembering that $\phi\sim\kappa$ (see~\e{n60b}), 
we see immediately that the enhanced integrability condition~\e{enhanced} implies that 
$f_0$ belongs to ${H}^{2,\phi}(\xR^2)$.
\end{proof}

\subsection{End of the proof}
The end of the proof of the three main results stated in the introduction is exactly similar to that of our previous work \cite{AN1,AN2}. As the last works are 
self-contained, to avoid repetitions, we 
will simply explain in this paragraph how we use weighted energy estimates to obtain uniform estimates.

$\bullet$ Proof of uniform estimates for small data in $W^{1,\infty}(\xR^2)\cap \dot H^2$. 

Let $\phi\equiv 1$ and $A_\phi=A,B_\phi=B,Z_\phi=Z$.	We have for $\beta_0=10^{-10}$ \begin{align*}
\frac{1}{2}\fract A(t)
&+\frac{B(t)}{1+ \lA\nabla f\rA_{L^\infty}^3}+|\log (\varepsilon)|^{-1}Z(t)\\&\leq C_0 A(t)^{\frac{1}{2}}(1+A(t))B(t)+C_0
\varepsilon^{2\beta_0} (1+A(t))^{2} (1+Z(t))^{\frac{4}{5}}.
\end{align*}
together with the following estimate for the slope
\begin{align*}
\sup_{\tau\in [0,t]}\lA \nabla f(\tau)\rA_{L^{\infty}}-\lA \nabla f_0\rA_{L^{\infty}}
&\leq C_0 
\int_{0}^{t}	B(\tau)\dtau\\
&\quad+ C_0\varepsilon^{\beta_0}\int_{0}^{t} A(\tau)\dtau+\varepsilon^{\beta_0/4}(1+\lA f_0\rA_{W^{1,\infty}}).
\end{align*} 
for any $t\leq |\log(\varepsilon)|$. 
Therefore, for $ (1+	\lA \nabla f_0\rA_{L^{\infty}})^3 	A(0)^{\frac{1}{2}}\ll 1$ and $\varepsilon\ll 1$, 
we have for any $t\leq |\log(\varepsilon)|$, 
\begin{align*}
	\fract A(t)+B(t)+|\log (\varepsilon)|^{-1}Z(t)\leq \varepsilon^{\frac{\beta_0}{2}},
\end{align*}
and 
\begin{align*}
	\sup_{\tau\in [0,t]}\lA \nabla f(\tau)\rA_{L^{\infty}}-&\lA \nabla f_0\rA_{L^{\infty}}
	\leq C_0 A(0)+\varepsilon^{\frac{\beta_0}{4}}.
\end{align*}

$\bullet$ Proof of uniform estimates for large data in $W^{1,\infty}(\xR^2)\cap \dot H^2$.

Consider the weight $\phi$ given by Lemma~\ref{L:critical}. 
We have for $\beta_0=10^{-10}$ \begin{align*}
		\frac{1}{2}\fract A_\phi(t)
		&+\frac{B_\phi(t)}{1+ \lA\nabla f\rA_{L^\infty}^3}+|\log (\varepsilon)|^{-1}Z_\phi(t)\\
		&\leq \varepsilon_1B_\phi(t)+\mathcal{F}\left(\frac{1}{\varepsilon_1}+ A_\phi(t)\right)+C_0
		\varepsilon^{2\beta_0} (1+A_\phi(t))^{2} (1+Z_\phi(t))^{\frac{4}{5}}.
	\end{align*}
On the other hand, 
	\begin{align*}
		\sup_{\tau\in [0,t]}\lA \nabla f(\tau)\rA_{L^{\infty}}-\lA \nabla f_0\rA_{L^{\infty}}
		&\leq C_0 
		\int_{0}^{t}	B_\phi(\tau)\dtau\\
&		+ C_0\varepsilon^{\beta_0}\int_{0}^{t} A_\phi(\tau)\dtau+C_0\varepsilon^{\beta_0/4}(1+\lA f_0\rA_{W^{1,\infty}})
	\end{align*} 
	for any $t\leq |\log(\varepsilon)|$. Therefore, there exists $T\ll 1$ such that for any $\varepsilon\ll 1$
	\begin{align*}
	\sup_{\tau\in [0,T]} A_\phi(\tau)+\int_{0}^{T} B_\phi(\tau)\dtau+|\log (\varepsilon)|^{-1}\int_0^T Z_\phi(\tau)\dtau\leq 10 A_\phi(0)
	\end{align*}
and 
\begin{align*}
	\sup_{\tau\in [0,t]}\lA \nabla f(\tau)\rA_{L^{\infty}}-&\lA \nabla f_0\rA_{L^{\infty}}\leq 10 A(0).
\end{align*}

$\bullet$ Proof of uniform estimates in $\dot{H}^{2,\log^a}(\xR^2)$.

Set $a=3/8$. It follows from Lemma~\ref{L:critical} that there exists 
$\widetilde{k}$ satisfying $\lim_{r\to +\infty}\widetilde{k}(r)=+\infty$ and such that 
$$
 \int_{\xR^2} |\xi|^4\big(\log(2+\la \xi\ra)\big)^{2a}\widetilde{k}(|\xi|)^2
\bla \hat{u}(\xi)\bra^2\dxi<+\infty.
$$
We consider the admissible weight $\kappa$ defined by $\kappa(r)=\log(2+r)^{a}\widetilde{k}(r)$ and then write
\be\label{lastfinal}
\begin{aligned}
	&\frac{1}{2}\fract A_\phi(t)
	+\frac{B_\phi(t)}{1+ \lA\nabla f\rA_{L^\infty}^3}+|\log (\varepsilon)|^{-1}Z_\phi(t)\\&\quad\quad\lesssim  A_\phi^{\frac{1}{2}}(1+A_\phi)B_\phi \left(\phi\left(\frac{B_\phi}{A_\phi}\right)\right)^{-1}+
	\varepsilon^{2\beta_0} (1+A_\phi)^{2} (1+Z_\phi(t))^{\frac{4}{5}},\nonumber
\end{aligned}
\ee
and  for 
$t\leq |\log(\varepsilon)|$
\begin{align*}
	\lA \nabla f(t)\rA_{L^{\infty}}\lesssim 1+\lA f_0\rA_{L^\infty}+\varepsilon^{\beta_0} \int_{0}^{t} (A_\phi(\tau)+ Z_\phi(\tau))\dtau+A_\phi\log (2+B_\phi)^{\frac{1}{8}}.
\end{align*}
From these estimates, we obtain local existence with $f_0\in \dot{H}^{2,\log^a}(\xR^2)$. To obtain global existence under a smallness assumption $A_{\log^a}(0)\ll 1$, we apply \e{lastfinal} where now $\kappa$ is simply 
$\kappa(r)=\log(2+r)^{a}$.

\appendix

\section{Weighted fractional laplacians}

\subsection{Proof of Proposition~\ref{P:3.5}}\label{A:2}

Proposition~\ref{P:3.5} is proved in \cite{AN1}. However, the proof in the latter reference 
is written only in dimension $d=1$ and for $1<s<2$. 
For the sake of completeness, we reproduce this proof to 
verify that it applies for $d=2$ and for any $0<s<2$. \\
$i)$ Notice that, for $h\in\xR^2$, the Fourier transform of $x\mapsto 2g(x)-g(x+h)-g(x-h)$ 
is given by $(2-2\cos(\xi\cdot h))\hat{g}(\xi)$. 
So, Plancherel's identity implies that
$$
\lA g\rA_{s,\kappa}^2
=\int_{\xR^2} I(\xi)\bla \hat{g}(\xi)\bra^2 \dxi,
$$
where
$$
I(\xi)=\frac{1}{\pi^2}\int_{\xR^2}  ( 1-\cos(\xi\cdot h))^2 \kappa^2\left(\frac{1}{|h|}\right)\frac{1}{|h|^{2+2s}}\dh.
$$
We have to prove that
\begin{equation}\label{Z30}
c|\xi|^{2s}\phi(|\xi|)^2\le I(\xi)\le C|\xi|^{2s}\phi(|\xi|)^2,
\end{equation}
for some constant $c,C$ independent of $\xi\in\xR$. 

We begin by proving the bound from above. Since 
$\la 1-\cos(\theta)\ra\le \min \{ 2,\theta^2\}$ for all $\theta\in \xR$, we have
\begin{align*}
	I(\xi)\lesssim \int \min \{ 1,(|\xi||h|)^4\} \kappa^2\left(\frac{1}{|h|}\right)\frac{dh}{|h|^{2+2s}}.
\end{align*}
By Lemma \ref{L:2.6}, one has for $\varepsilon_0>0$
\begin{align*}
	(\min \{ 1,|\xi||h|\})^{\varepsilon_0} \kappa^2\left(\frac{1}{|h|}\right)\lesssim_{\varepsilon_0} \kappa^2(|\xi|).
\end{align*}
It follows for $\varepsilon_0<<1$
\begin{align*}
	I(\xi)\lesssim_{\varepsilon_0}  \kappa^2(|\xi|)\int \min \{ 1,(|\xi||h|)^{4-\varepsilon_0}\} \frac{dh}{|h|^{2+2s}}\lesssim_{\varepsilon_0} |\xi|^{2s}\phi(|\xi|)^2.
\end{align*}
The proof of the lower bound is straightforward. 

Recalling the notation $\check{\xi}=\xi/|\xi|$, and integrating in polar coordinate, we find that
\begin{align*}
I(\xi)
&=\frac{2}{\pi}\int_0^{+\infty}(1-\cos(\la \xi\ra r))^2  \kappa^2\left(\frac{1}{r}\right)\frac{1}{r^{1+2s}}\dr
\gtrsim \int_{\frac{2\pi}{5}\le r\la \xi\ra \le \frac{3\pi}{5}}\kappa^2(r)\frac{\dr}{r^{1+2s}}\\
&\gtrsim \bigg(\int_{\frac{\pi}{3}\le r\la \xi\ra \le \frac{3\pi}{5}}\dr\bigg)\kappa^2\left(\frac{1}{|\xi|}\right)\la \xi\ra^{1+2s}
\gtrsim \kappa^2\left(\frac{1}{|\xi|}\right)\la \xi\ra^{2s}.
\end{align*}

This completes the proof of \e{Z30}. The proof of statement $ii)$ is similar.

\subsection{Interpolation estimates}
In this paragraph, we collect various interpolation estimates. These results are not new: we follow closely the proofs of similar results in \cite{AN1} whhi

Consider a function $f$ and a weight $\phi$. 
We want to control the homogeneous Sobolev norms in terms of the following 
quantities
\be\label{n67}
\begin{aligned}
A_\phi&=\blA \D^{2,\phi}f\brA_{L^2}^2,\quad
B_\phi&=\blA \D^{\frac52,\phi}f\brA_{L^2}^2,\quad
\mu_\phi&=\left(\phi\left(\frac{B_\phi}{A_\phi}\right)\right)^{-1} .
\end{aligned}
\ee
\begin{lemma}\label{L:3.5}
$i)$ For all $s\in (2,5/2)$, 
there exists a positive constant $C$ such that,
\begin{align}
\lA f\rA_{\dot H^s}\le C\mu_\phi A_\phi^{\frac52-s}  B_\phi^{s-2}.
\end{align}
$ii)$ Assume that $\phi$ is of the form~\e{n10-ter} where $\kappa$ is an admissible weight 
satisfying $\kappa(r)\ge \log(4+r)^a$ for some $a\ge 0$. Then
\be
\Vert \nabla f\Vert _{L^\infty}\lesssim 1+ \lA f\rA_{L^\infty}+A_\phi\log (2+B_\phi)^{\frac{1-2a}{2}}.\label{linf'}
\ee

\end{lemma}
\begin{proof}We follow closely the proofs of similar results in \cite{AN1}.

$i)$ Let $\lambda >0$. 
Decompose the frequency space into low and high frequencies, at the frequency threshold $\la \xi\ra=\lambda$, to obtain
\begin{align*}
\lA f\rA_{\dot H^s}^2&\les 
\int_{\xR^2}|\xi|^{2s}|\hat{f}|^2 \dxi
=\int_{|\xi|\le \lambda}|\xi|^{2s}|\hat{f}|^2 \dxi
+\int_{|\xi|> \lambda} |\xi|^{2s}|\hat{f}|^2 \dxi\\
&\les\int_{|\xi|\le \lambda} \frac{\la \xi\ra^{2s-4}}{\kappa(|\xi|)^2}|\xi|^{4}
\phi(|\xi|)^2|\hat{f}|^2 \dxi
+\int_{|\xi|> \lambda} \frac{\la \xi\ra^{2s-5}}{\kappa(|\xi|)^2}|\xi|^{5}\phi(|\xi|)^2|\hat{f}|^2 \dxi.
\end{align*}
It follows that
$$
\lA f\rA_{\dot H^s}^2\lesssim \lambda^{2s-4} (\kappa(\lambda))^{-2} \Vert \D^{2,\phi}f\Vert_{L^2}^2+\lambda^{2s-5}  (\kappa(\lambda))^{-2} \blA\D^{\frac52,\phi}f\brA_{L^2}^2.
$$
Chose $\lambda=\Vert \D^{\frac52,\phi}(f)\Vert _{L^2}^2/\Vert \D^{2,\phi}(f)\Vert _{L^2}^2$ to get the wanted result. 

$ii)$ 
One has 
\begin{align*}
\lA\nabla f\rA _{L^\infty}
&\lesssim \lA f\rA_{L^\infty}
+\Vert \nabla( f-f\star \chi_1)\Vert _{L^\infty}\\
&\lesssim \lA f\rA_{L^\infty}+\int |\xi|\min\{|\xi|,1\}|\hat f(\xi)| \dxi.
\end{align*}
It is easy to see that 
\begin{align*}
	&\int_{|\xi|\leq \lam} |\xi|\min\{|\xi|,1\}|\hat f(\xi)| d\xi\lesssim A\log (2+\lam)^{\frac{1-2a}{2}},\\&\int_{|\xi|> \lam} |\xi|\min\{|\xi|,1\}|\hat f(\xi)| d\xi\lesssim  B\lam^{-1/2}.
\end{align*}
Choosing $\lam= (B+1)^2$
\begin{equation*}
	\Vert \nabla f\Vert _{L^\infty}\lesssim 1+ \lA f\rA_{L^\infty}+A\log (2+B)^{\frac{1-2a}{2}},
\end{equation*}
we obtain~\e{linf'}. This completes the proof. 
\end{proof}

\section*{Acknowledgments} 

\noindent The authors want to thank Beno{\^\i}t Pausader for suggesting us to consider initial data which do not belong to $L^2(\xR^2)$ 
as well as for his stimulating 
comments.

\noindent  T.A.\ acknowledges the SingFlows project (grant ANR-18-CE40-0027) 
of the French National Research Agency (ANR).  Q-H.N.\ 
is  supported  by the ShanghaiTech University startup fund.

\vfill
\begin{flushleft}
\textbf{Thomas Alazard}\\
Universit{\'e} Paris-Saclay, ENS Paris-Saclay, CNRS,\\
Centre Borelli UMR9010, avenue des Sciences, 
F-91190 Gif-sur-Yvette\\
France.

\vspace{1cm}

\textbf{Quoc-Hung Nguyen}\\
ShanghaiTech University, \\
393 Middle Huaxia Road, Pudong,\\
Shanghai, 201210,\\
China

\end{flushleft}

\end{document}